\documentclass[12pt]{amsart}
 
 \usepackage{amssymb}
 \usepackage{colonequals}
 \usepackage{wasysym}
\usepackage{array}
\usepackage{amsthm}
\theoremstyle{plain}
\usepackage{booktabs}
\usepackage{amsmath}
\usepackage{amsxtra}
\usepackage{mathrsfs}
\usepackage[all]{xy}
\usepackage{stmaryrd}
\usepackage{blindtext}
\usepackage{pdflscape}
\usepackage[usenames, dvipsnames]{color}
\usepackage{color}
\usepackage{wrapfig}
\usepackage{tcolorbox}
\usepackage{graphicx}
\usepackage{subfigure}
\usepackage{mathtools}
\usepackage{tikz}
\usepackage{datetime}
\usepackage{manfnt}
\usepackage{bm}
\usepackage[shortlabels]{enumitem}

\makeatletter
\@namedef{subjclassname@2020}{%
  \textup{2020} Mathematics Subject Classification}
\makeatother

\usepackage{algorithm}
\usepackage{algpseudocode}

\usepackage{array}

\def\R{{{\mathbb R}}}
\def\Q{{{\mathbb Q}}}
\def\Z{{{\mathbb Z}}}

\def\C{{{\mathbb C}}}

\def\p{{\mathfrak{p}}}
\def\q{{{\mathfrak q}}}

\def\deg{{\text{deg}}}

\def\Hom{{\text{Hom}}}

\def\Cl{{{\operatorname{Cl}}}}
\def\GL{{{\operatorname{GL}}}}
\def\SL{{{\operatorname{SL}}}}

\def\SL{{\text{SL}}}

\newcommand{\nr}{\operatorname{nr}}

\newtheorem{prop}{Proposition}
\numberwithin{prop}{section}

\newtheorem{thm}{Theorem}
\newtheorem{defn}[prop]{Definition}
\newtheorem*{thm*}{Theorem}

\newtheorem{cor}[prop]{Corollary}
\newtheorem{lem}[prop]{Lemma}
\theoremstyle{remark}

\numberwithin{remark}{section}

\theoremstyle{definition}
\newtheorem{example}[prop]{Example}
\usepackage{fullpage}

\begin{document}

\title{Type numbers of locally tiled orders in central simple algebras}

\author[A. Babei]{Angelica Babei}
\address{Pavillon Andr\'{e}-Aisenstadt, Centre de Recherches Math\'{e}matiques\\ Universit\'{e} de Montr\'{e}al\\
Montr\'{e}al, Qu\'{e}bec H3T 1J4\\
Canada}
\email{babeiangelica@gmail.com}

\begin{abstract} Let $A$ be a central simple algebra over a number field $K$ with ring of integers $\mathcal{O}_K$, such that either the degree of the algebra $n \ge 3$, or $n=2$ and $A$ is not  a totally definite quaternion algebra. Then strong approximation holds in $A$, which allows us to describe the genus of an $\mathcal{O}_K$-order $\Gamma \subset A$ in terms of idelic quotients of the field $K$. We consider orders $\Gamma$ that are tiled at every finite place $\nu$ of $K$ and use the Bruhat-Tits building for $\SL_n(K_\nu)$ to give a geometric description for the local normalizers of $\Gamma$. We also give explicit formulas and algorithms to compute the type number of $\Gamma$. Our results generalize work of Vign\'{e}ras \cite{arithmetique} for orders in higher degree central simple algebras.
\end{abstract}



\maketitle

\section{Introduction}

Type numbers of orders in central simple algebras have been investigated in different contexts. The initial interest has been in finding type numbers of maximal and Eichler orders in totally definite quaternion algebras, with fomulas given first  by Deuring \cite{deuring}, and subsequently by Eichler \cite{zurzahlen}, Peters \cite{ternare} and Pizer \cite{type, pizer2}. The rich arithmetic structure of quaternion orders and the type number formulas gave rise to various applications in areas such as   the theory of ternary quadratic forms \cite{shim}, computing  traces of Brandt matrices for classical modular forms \cite{theta}, or computing spaces of Hilbert modular cusp forms \cite{unram}. The case of Eichler orders in  not totally definite quaternion algebras is also of interest, since such orders give rise to Shimura curves. Type numbers of such orders have been computed by Vign\'{e}ras in \cite{arithmetique} using  strong approximation, a tool on which we will also rely on in this article.

Let $A$ be a central simple algebra of degree $n \ge 2$  over a number field $K$ with ring of integers $\mathcal{O}_K$, such that either $ n \ge 3$, or $A$ is not a totally definite quaternion algebra. Then strong approximation holds in $A$. Let $\Gamma$ be an $\mathcal{O}_K$-order in $A$. The \textit{type number} of $\Gamma$ is the number of isomorphism classes of orders that are locally isomorphic to $\Gamma$, which constitutes the \textit{genus} of $\Gamma$. We will denote the type number by   $G(\Gamma)$.

We follow the conventions in \cite{local}, where the authors investigate  maximal orders. In particular, they apply strong approximation and express the arithmetic of the global order in terms of idelic arithmetic over the field $K$ in the following way.
The reduced norm maps $\nr_{A/K}:A \rightarrow K$ and $\nr_{A_\nu/K_\nu}:A_\nu \rightarrow K_\nu$ at all places $\nu$ of $K$ induce norm maps on the ideles $\nr:J_A \rightarrow J_K$, where $\nr((a_\nu)_\nu)=(\nr_{A_\nu/K_\nu}(a_\nu))_\nu$.  Consider a maximal order $\Lambda \subseteq A$, then  $\Lambda_\nu$ is maximal at all finite places. Denote the normalizer of $\Lambda_\nu$ by $\mathcal{N}(\Lambda_\nu)$, and  the restricted product $\prod'_{\nu} \mathcal{N}(\Lambda_\nu):=J_A\cap \prod_{\nu}\mathcal{N}(\Lambda_\nu)$. Then the type number $G(\Lambda)$ is given by the number of double cosets $ A^\times\backslash J_A/\prod'_{\nu} \mathcal{N}(\Lambda_\nu)$. As a consequence of strong approximation,   the reduced  norm induces  a bijection \begin{equation*}
 \nr: A^\times\backslash J_A/\sideset{}{'}\prod_{\nu} \mathcal{N}(\Lambda_\nu) \rightarrow J_K/K^\times \nr(\sideset{}{'}\prod_{\nu} \mathcal{N}(\Lambda_\nu)).
\end{equation*}

In particular, when $\Lambda_\nu$ is maximal in $A_\nu\cong M_{n_\nu}(D_\nu)$ where $D_\nu$ is a division algebra over $K_\nu$,  $\nr(\Lambda_\nu)=(K_\nu^\times)^{n_\nu}\mathcal{O}_\nu^\times$. A natural question would be to ask what kind of groups we can get  when $\Lambda_\nu$ is nonmaximal. In this article, we study a class of orders $\Gamma$ for which $\nr(\mathcal{N}(\Gamma_\nu))=(K_\nu^\times)^{d_\nu}\mathcal{O}_\nu^\times$ with $d_\nu|n_\nu$ and describe this exponent geometrically.

In particular, we investigate type numbers of orders that are \textit{tiled} at all finite places, and we call such orders \textit{everywhere locally tiled}. When it is clear that we work in the global context, we simply call such global orders ``locally tiled".  Background information for (local) tiled orders can be found in Section \ref{prelim}. Tiled orders are of interest to us for a few reasons. First, they are a class more general than maximal and hereditary orders. Second, we can use a combinatorial and geometric framework to investigate their algebraic properties.


To compute the type number $G(\Gamma)$ of an everywhere locally tiled order $\Gamma$, we apply strong approximation and use the same idelic quotient as in the equation above.  In order to describe the local results, we switch to the local notation used in Section \ref{tiled loc}. To avoid confusion, we only use local notation in Section \ref{tiled loc}, and return to global notation in Section \ref{typs}.  In the local setting, when we refer to a tiled order $\Gamma$,  we mean a local order $\Gamma \subseteq M_{r}(D)$, where $D$ is a division algebra over a  non-archimedean local field $k$.  We assume $char(k)=0$, and denote the valuation ring of $k$ by $R$. As in \cite{poly}, we  can associate to $\Gamma$ a convex polytope $C_\Gamma$ in an apartment $\mathcal{A}$ in the building for $\SL_{r}(D)$. Isomorphic orders will have geometrically congruent polytopes, and we can partition such polytopes into equivalence classes. Each class consists of polytopes that can be connected through reflections across hyperplanes in $\mathcal{A}$, which we call \textit{reflection equivalent}. Additionally, this equivalence relation can be represented using algebraic invariants of the tiled order. We get equivalence classes $[\Gamma_0], [\Gamma_1], \dots, [\Gamma_{r-1}]$, and can connect the set of such equivalence classes with $\nr(\mathcal{N}(\Gamma))$.

\begin{thm*}
Let $\Gamma$ be a tiled order with correponding reflection classes $[\Gamma_i]$. Then the following are equivalent:

\begin{enumerate}[(a)]

\item There are $d$ distinct equivalence classes.

\item $d$ is the smallest among $\{1, 2, \dots, r\}$ such that  $[\Gamma_s]=[\Gamma_t]$ whenever $s \equiv t \pmod{d}$.

\item $d$  is the smallest among $\{1, 2, \dots, r\}$ such that  $[\Gamma_0]=[\Gamma_{d}]$.

\item $\nr(\mathcal{N}(\Gamma))=(k^\times)^{d} R^\times$.
\end{enumerate}

\end{thm*}

This theorem allows us to find the number of reflection classes for any (local) tiled order $\Gamma$ as described in Algorithm \ref{alg1} in Section \ref{tiled loc}. While the general algorithm requires some knowledge about $\mathcal{N}(\Gamma)$, there is one particular case,  when the local algebra $M_{p}(D)$ has $p$  a prime number, which does not require finding the normalizer $\mathcal{N}(\Gamma)$. This particular case is described in Algorithm \ref{alg2} of Section \ref{tiled loc}. 

We return to global notation and compute type numbers in Section \ref{typs} by expressing the idelic cosets in terms of class groups. 

\begin{thm*}
Let $A$ be a central simple algebra of degree $n \ge 2$  over a number field $K$  such that either $ n \ge 3$, or $A$ is not a totally definite quaternion algebra. Let  $\Gamma$ be an everywhere  locally tiled order in $A$.  Let $\Omega$ be the set of real ramified primes in $A$, $\Cl_\Omega(K)$ be the ray class group modulo the real places in $\Omega$,  $S=S_\infty -\Omega$ and $T=\{\p \text{ finite }:\nr(\mathcal{N}(\Gamma_\p))=(K_\p^\times)^{d_\p}R_\p^\times, d_\p  \ne n \}$. For each place $\p \in T$, pick a prime $\q_\p$ such that $ [\q_\p]=[\p^{d_\p}] \text{ in } \Cl_\Omega(K)$ and let $\hat{T}=\{ \q_\p : \p \in T\} \cup S$. Then 
\[
G(\Gamma)=\#\Cl_{\hat{T}, \Omega}(K)/\Cl_{\hat{T}, \Omega}(K)^n,
\] where $\Cl_{\hat{T}, \Omega}(K)=\Cl_\Omega(K)/\langle [\q_\p] : \p\in T\rangle$.
\end{thm*}

Together with the algorithms in Section \ref{tiled loc}, we can use the above theorem to compute type numbers of any everywhere  locally tiled order, and we illustrate it with an example.

Many of the results in this article can be generalized to algebras over general global fields, however there are various cases that require caution. For example, not all such algebras have strong approximation, in which case other sets of tools would be necessary for finding type numbers. Some steps towards a generalization can be found in Brzezinski \cite{gentype}.  On the other hand, we could also look at $\mathcal{O}$-orders in $A$ where $\mathcal{O}$ is an arbitrary order in the number field $K$, but their associated class groups would require extra care.

\section{Preliminaries}
\label{prelim}

\subsection{Class groups and ideles} Let $A$ be a central simple algebra over a number field $K$ such that either the degree of the algebra $n \ge 3$ or $n=2$ and $A$ is not a totally definite quaternion algebra; then strong approximation holds in $A$. Denote the ring of integers of $K$ by  $\mathcal{O}_K$ and the set of places of $K$ by $\text{Pl}(K)$. Let $\Gamma$ be an $\mathcal{O}_K$-order in $A$. We denote by $K_\nu$ and $\mathcal{O}_\nu$ the completions of $K$, and respectively $\mathcal{O}_K$,  at a place $\nu$ of $K$, and let $A_\nu\coloneqq K_\nu \otimes_K A$ and $\Gamma_\nu \coloneqq  \mathcal{O}_\nu \otimes_R \Gamma$.  If $\nu$ is an infinite place, we set $\mathcal{O}_\nu\coloneqq K_\nu$ and  $\Gamma_\nu \coloneqq A_\nu$.

Given a finite set of places $S$ of $K$, we define the set of  $S$-ideles by \[J_{K,S}\coloneqq \prod_{\nu \in S} K_\nu^\times \prod_{\nu \not \in S} \mathcal{O}_\nu^\times. \] 
 We denote  the ideles of $K$ by \[J_K\coloneqq  \mathop{\bigcup_{S \subseteq \text{Pl}(K)}}_{ S \text{ finite}}  J_{K,S} \subseteq \prod_\nu K_\nu^\times.\] We also write $
J_K = \sideset{}{'}\prod_{\nu} K_\nu^\times,$
where $\prod_\nu'$ is the restricted product over the places $\nu$ of $K$.
Similarly, we can define $J_A$ the ideles of $A$. We have reduced norm maps $\nr_{A/K}:A \rightarrow K$ and $\nr_{A_\nu/K_\nu}:A_\nu \rightarrow K_\nu$, which induce $\nr:J_A \rightarrow J_K$ where $\nr((a_\nu)_\nu)=(\nr_{A_\nu/K_\nu}(a_\nu))_\nu$. Denote the normalizer of $\Gamma_\nu$ by $\mathcal{N}(\Gamma_\nu)=\{\xi \in A_\nu^\times \,|\, \xi \Gamma_\nu\xi^{-1}=\Gamma_\nu\}$, and  the restricted product $\prod'_{\nu} \mathcal{N}(\Gamma_\nu)\coloneqq J_A\cap \prod_{\nu}\mathcal{N}(\Gamma_\nu)$.

Consider the set of $\mathcal{O}_K$-orders in $A$ locally isomorphic to $\Gamma$; this is the \textit{genus} of $\Gamma$. By the Skolem-Noether theorem, this set consists of $\mathcal{O}_K$-orders $\Lambda\subset A$ such that  $\Lambda_\nu = \xi_\nu \Gamma_\nu \xi_\nu^{-1}$ for some $ \xi_\nu \in A^\times_\nu$ at all finite places $\nu$ of $K$. Local isomorphisms don't necessarily lift to global isomorphisms, and we wish to investigate the isomorphism classes in the genus of $\Gamma$.  Since by the Skolem-Noether theorem, $\Lambda$ and $\Gamma$ are isomorphic if and only if $\Lambda=\xi \Gamma \xi^{-1}$ by some $\xi \in A^\times$, the isomorphism classes correspond to the double cosets $A^\times\backslash J_A/\prod'_{\nu} \mathcal{N}(\Gamma_\nu)$. We denote the cardinality of the double cosets  by $G(\Gamma)$,  which  is also known as the \textit{type number} of $\Gamma$. Our main goal is to compute $G(\Gamma)$ in the case where each completion $\Gamma_\nu$ is \textit{tiled}. Tiled orders generalize maximal and hereditary orders; we define them in Section \ref{tiled}.

As a  consequence of strong approximation, the reduced norm map  induces a bijection 

\begin{equation}
\label{types}
 \nr: A^\times\backslash J_A/\sideset{}{'}\prod_{\nu} \mathcal{N}(\Lambda_\nu) \rightarrow K^\times \backslash J_K/\nr(\sideset{}{'}\prod_{\nu} \mathcal{N}(\Lambda_\nu)) \cong J_K/K^\times \nr(\sideset{}{'}\prod_{\nu} \mathcal{N}(\Lambda_\nu)),
\end{equation}

\noindent (for more background, see \cite[Corollary 28.4.8]{quatbook} and \cite[Theorem 3.1]{local}), where the codomain has the structure of an abelian group.

 In order to find the number of cosets as in the equation above, we need to connect such idelic cosets with class groups of $K$.
We follow sections I.6, VII.1 and VII.3 in \cite{lang}. Let $\Cl(K)$ be the class group of $K$.  Then $\Cl(K) \cong J_K/K^\times J_{K, S_\infty}$, where $S_\infty$ is  the set of infinite places of $K$. It is well known that the class group of a number field is finite; let $h(K)\coloneqq  \#\Cl(K)$ be the class number of $K$.

We also introduce more general class groups.  Let $\Omega$ be a subset of the real places of $K$, and $S$ a finite set of places of $K$ such that $S_\infty \subseteq S \cup \Omega$ and $S \cap \Omega = \emptyset$. Define $$ J_{K,  S, \Omega} \coloneqq \prod_{\nu \in \Omega} \R_+^\times \prod_{\nu \in S} K_\nu^\times \prod_{\nu \not \in S \cup \Omega} \mathcal{O}^\times_\nu.$$ Note that in our previous notation, $J_{K, S_\infty}= J_{K, S_\infty, \emptyset}$, so we drop the subscript $\Omega$ when $\Omega$ is empty. We define the $(S, \Omega)$-class group of $K$ by \[\Cl_{S, \Omega}(K)\coloneqq J_K/K^\times J_{K, S, \Omega}.\]

In the particular case where $S \cup \Omega = S_\infty$ and therefore $S$ contains no finite places, $\Cl_{S, \Omega}$ is uniquely determined by $\Omega$. For notational convenience, we write in this case $\Cl_\Omega \colonequals \Cl_{S, \Omega}(K)$. The group $\Cl_\Omega(K)$ is also known as the $\Omega$-ray class group of $K$, and can also be realized the following way.  Let $K_\Omega \coloneqq \{a \in K: \nu(a)>0 \,\, \text{for all} \,\, \nu \in \Omega\}$ be the subset of $K$ consisting of elements of $K$ that are positive at all places in $\Omega$. Then  $\Cl_\Omega(K)= I_K/P_\Omega$, where $I_K$ is the set of fractional ideals of $\mathcal{O}_K$ and $P_\Omega=\{ (a): a \in K^\times_\Omega\}$ is the set of principal ideals generated by elements of $K_\Omega$. 

We have a global interpretation for the groups $\Cl_{S, \Omega}(K)$ as well. Let   $T=\{ \p \in S: \p \text{ is  finite}\}$. Note that  $J_{K, S\setminus T, \Omega} \subseteq J_{K,S, \Omega}$, so we have a surjective homomorphism $\Cl_{S \setminus T, \Omega}(K) = \Cl_{\Omega}(K) \twoheadrightarrow \Cl_{S, \Omega}(K)$ with kernel $\prod_{\p \in T} (\dots, 1, K_\p^\times, 1, \dots) K^\times J_{K, S\setminus T, \Omega}$. Since the uniformizers $\pi_\p$ generate each $K_\nu^\times$ and each coset $(\dots, 1, \pi_\p, 1, \dots) K^\times J_{K, S\setminus T, \Omega}$ corresponds to the ideal class $[\p] \in \Cl_\Omega$, we get  $$\Cl_{S, \Omega}(K)\cong \Cl_{\Omega}(K)/\langle [\p] : \p \in T\rangle.$$ Therefore,  each class group $\Cl_{S, \Omega}(K)$ can be realized equivalently either globally as a quotient of the fractional ideals of $\mathcal{O}_K$, or locally as an idelic quotient. We will use both characterizations in our calculations.

\subsection{Central simple algebras over local fields}
\label{local}

We now proceed locally and introduce the corresponding notation, used primarily in Section \ref{tiled loc}. Let $A$ be a central simple algebra over a non-archimedean local field $k$  ($\text{char} \, k =0$) with a valuation $v$ and valuation ring $R$, unique maximal ideal $P$ and uniformizer $\pi$, such that the residue field $\overline{R}\coloneqq R/P$ is finite of size $q$. Following Chapter 5 in \cite{reiner}, let $V$ be a minimal right ideal of $A$, and let $D=\Hom_A(V,V)$. Then we can view $_DV_A$ as a bimodule, where given any left $D$-basis $\{v_1, \dots, v_n\}$ for $V$, we have the action $$(v_1, \dots, v_n)a=(\alpha_{ij})(v_1, \dots, v_n), \quad (\alpha_{ij})\in M_n(D)$$ for any $a \in A$, and from now on we identify $A$ with $M_n(D)$ and $V$ with $D^n$. By the Artin-Wedderburn theorem, $D$ is a central division algebra  over $k$ of some degree $m$, so $\deg(A)=mn$.  Let $\Delta$  be the unique maximal $R$-order in $D$, equipped with a prime element $\bm{\pi}$ such that $\bm{\pi}^m=\pi$.

Motivated by Equation (\ref{types}), we want to investigate the reduced norm map $\nr_{M_n(D)/k}$.  We start with $\nr_{D/k}$, as discussed in \cite[ Section 14]{reiner}. In particular, $D$ contains a maximal subfield $W$ (with valuation ring $R_W$), which is an unramified extension of $k$.  Then  $W=k(\omega)$, where $\omega$ is a $(q^{m}-1)$th root of unity. In addition, $W$ is a splitting field for $D$, so  $ D \otimes_k W \cong M_{m}(W)$ and there is  an embedding (see \cite[p. 18]{reiner}) $\mu: D \hookrightarrow M_{m}(W)$ which induces the norm map $\nr_{D/k}(x)=\det(\mu(x))$. Denoting by $N_{W/k}$ is the regular norm map, this embedding gives
 \[ \nr_{D/k}(\bm{\pi}) =(-1)^{m-1}\pi, \,\, \text{and} \,\,  \nr_{D/k}(\alpha)=N_{W/k}(\alpha)\,\,  \text{for all}\,\,\alpha \in W.\]

To obtain reduced norms for the whole algebra $M_n(D)$ over $k$, we note that the map $\mu$ induces an embedding $\widehat{\mu}: M_{n}(D) \hookrightarrow M_{mn}(W)$ where $(a_{ij}) \in M_{n}(D)$ and $(a_{ij}) \mapsto (\mu(a_{ij}))$.  By  \cite[p. 282]{reiner},
\begin{equation}
\label{defnnrd}
\nr_{M_{n}(D)/k}(x)=\det(\widehat{\mu}(x)).
\end{equation}

In particular, by embedding $D$ on the diagonal in $M_{n}(D)$, we get 
\begin{equation}
\label{nrD} \nr_{M_{n}(D)/k}(x)=(\nr_{D/k}(x))^{n} \qquad \text{for all} \quad x \in D
\end{equation}

A couple of notes on  $\nr_{D/k}(\Delta)$ and $\nr_{M_n(D)/k}(M_n(\Delta))$. By  \cite[p.319]{neukirch}, $N_{W/k}$ maps the units of $R_W$ onto $R^\times$.  By \cite[Theorem 14.4]{reiner}, we also have that $\Delta=R[\omega, \bm{\pi}]=R_W[\bm{\pi}]$, and therefore $\mu(\Delta) \subseteq M_m(R_W)$ and  $\widehat{\mu}(M_n(\Delta)) \subseteq M_{mn}(R_W)$. Putting it all together, we get 
\begin{equation}
\label{deltanrd}
\nr_{D/k}(\Delta)=R \qquad \text{and} \qquad  \nr_{M_n(D)/k}(M_n(\Delta))=R.
\end{equation}

Finally, we also have have a normalized valuation $v_D$ on $D$, such that (see \cite[Equation 13.1]{reiner})
\begin{equation}
\label{valuations}
 v_D(a)=\frac{1}{m}v((\nr_{D/k}a)^m)=v(\nr_{D/k}a).
\end{equation}

We connect the two valuations with the reduced norm $\nr_{M_n(D)/k}$ in the matrix algebra using the Dieudonn\'{e} determinant $\det:\GL_n(D) \rightarrow D^\times/[D^\times, D^\times]$. We define $\SL_n(D)$ to be the kernel of the Dieudonn\'{e} determinant.
\begin{lem}
\label{teqn}
Let $v$ and $v_D$ be (normalized) valuations on $k$ and $D$, and let $x \in GL_n(D)$. Then
$ v_D(\det(x))=v(\nr_{M_n(D)/k}(x)).$
\end{lem} 

\begin{proof} By Equation (\ref{defnnrd}),  $v(\nr_{M_n(D)/k}(x))=v(\det(\widehat{\mu}(x)))$. Since both the reduced norm map and the Dieudonn\'{e} determinant are multiplicative (\cite[Theorem 2.2.5]{ktheory}),  it suffices to prove the equality for elementary matrices. The claim is clear for row-addition and row-swapping matrices.  It remains to check the claim for row-multiplication matrices. Again by multiplicativity, it suffices to consider the diagonal matrix $d =\text{diag}(y,1,1\dots,1)$, for $y \in D$. Then $v_D(\det(d))=v_D(y)$, and $v(\nr_{M_n(D)/k}(d))=v(\nr_{D/k}(y))$. By Equation (\ref{valuations}), the two valuations are equal.
\end{proof}

Now that we have the connection between the reduced norm of a matrix and its Dieudonn\'{e} determinant, we introduce the \textit{type} of a matrix. Note that while we encounter both type numbers of orders and types of matrices in this article, the two terms are already deeply ingrained in the literature but not in any way connected to each other.

\begin{defn}
Let $g \in GL_n(D)$.  We define the \textbf{type} of the matrix $g$ by \[t(g)\coloneqq v_D(\det(g)) \pmod{n}.\] 
\end{defn}

Consider monomial matrices in $\GL_n(D)$, which are of the form $\xi=(\bm{\pi}^{\alpha_i}\delta_{\sigma(i)j})$, where $\delta_{ij}$  is the Kronecker delta and $\sigma$ a permutation in $S_n$. Then $\xi$ is a product of the diagonal matrix $d=(\bm{\pi}^{\alpha_i}\delta_{ij})$ and the permutation matrix $p_\sigma=(\delta_{\sigma(i)j})$, so by \cite [Theorem 2.2.5]{ktheory} its Dieudonn\'{e} determinant  is 
$\displaystyle \det(\xi)=\det(d)\det(p_\sigma)=\text{sgn}(\sigma) \bm{\pi}^{\sum_{i=1}^n \alpha_i}.$ Therefore, we have 
\[
t(\xi)\equiv \sum_{i=1}^n \alpha_i \pmod{n}.
\]

By Lemma \ref{teqn}, $t(g)\equiv v(\nr_{M_n(D)/k}(g)) \pmod{n}$, and  motivated by Equation (\ref{types}) we will use the type of a matrix as a proxy for its reduced norm in type number computations.

\subsection{The building for $\SL_n(D)$}

 We introduce the building-theoretic framework used throughout, following Chapter 9 in \cite{ronan}.  In particular, we construct the affine building for $\SL_n(D)$. we say that two full $\Delta$-lattices  $L_1$ and $L_2$ in $V$ are \textit{homothetic} if $L_1=aL_2$ for some $a\in D$. Homothety of lattices is an equivalence relation, and we denote the homothety class of $L$ by $[L]$.  The vertices in the building  are the homothety classes of lattices in $D^n$, and  there is an edge between two vertices if there are lattices $L_1$ and $L_2$ in their respective homothety classes such that $\bm{\pi} L_1 \subsetneq  L_2\subsetneq L_1$. The vertices of an $\ell$-simplex correspond to chains of lattices of the form $\bm{\pi} L_1 \subsetneq L_2 \subsetneq \dots \subsetneq L_{\ell+1} \subsetneq L_1$. The maximal $(n-1)$-simplices are called \textit{chambers}. 

To each frame of lines generated by a basis $\{v_1, v_2, \dots, v_n\}$, we have an associated subcomplex of the affine building for $\SL_n(D)$, called an \textit{apartment}. The vertices of the apartment  are homothety classes of lattices  of the form $L=\Delta \bm{\pi}^{m_1} v_1\oplus \Delta\bm{\pi}^{m_2}v_2\oplus \dots \oplus \Delta\bm{\pi}^{m_n}v_n, \, m_i \in \Z$, which we encode by $[L]=[m_1, m_2, m_3, \dots, m_n] = [0,m_2-m_1,\dots, m_n-m_1]$. 

 Each apartment is an $(n-1)$-complex and a tessellation of $\R^{n-1}$, with hyperplanes $x_i-x_j=\mu_{ij}$ for $i\in \{1,2,\dots, n\}$, where at the intersection of $(n-1)$ pairwise non-parallel hyperplanes we obtain a homothety class as seen in Example \ref{ex0}. Since we can switch between apartments by conjugating the basis, and conjugation does not change the reduced norm of $\mathcal{N}(\Gamma)$, from now on we fix the apartment $\mathcal{A}$ associated to the standard basis $\{e_1,e_2, \dots, e_n\}$.

\begin{example}
\label{ex0} In Figure \ref{apt}, we see a piece of the apartment in the building for $\SL_3(D)$, where lines in the apartment correspond to \textit{hyperplanes}, and are given by equations of the form $x_s-x_t=\mu, \mu \in \Z$.
\end{example}

\begin{figure}[h]
\label{apt}

\smallskip

\begin{tikzpicture}[scale=0.8, every node/.style={transform shape}]

\draw (0,0) -- +(1,1.75) -- (2,0) -- +(1,1.75)-- (4,0) -- +(1,1.75)-- (6,0) -- +(1,1.75) -- (8,0) -- +(1,1.75) -- (10,0) -- +(1,1.75) -- (12,0) -- +(1,1.75) -- (14,0);

\draw (0,1.75) -- (14, 1.75);

\draw (0,0) -- (14, 0);

\draw (0,0) -- (1,-1.75) -- (2,0) -- (3,-1.75)-- (4,0) -- (5,-1.75)-- (6,0) -- (7,-1.75) -- (8,0) -- (9,-1.75) -- (10,0) -- (11,-1.75) -- (12,0) -- (13,-1.75) -- (14,0);

\draw (0,-1.75) -- (14, -1.75);

\draw (0,-3.5) -- +(1,1.75) -- (2,-3.5) -- +(1,1.75)-- (4,-3.5) -- +(1,1.75)-- (6,-3.5) -- +(1,1.75) -- (8,-3.5) -- +(1,1.75) -- (10,-3.5) -- +(1,1.75) -- (12,-3.5) -- +(1,1.75) -- (14,-3.5);

\draw (0,-3.5) -- (14, -3.5);

\node[above right] at (0,0)  {\fontsize{7}{5}\selectfont$ \,\, [0,2,-3]$};
\node[above right] at (2,0)  {\fontsize{7}{5}\selectfont$ \,\,  [0,2,-2]$};
\node[above right] at (4,0)  {\fontsize{7}{5}\selectfont$ \,\,  [0,2,-1]$};
\node[above right] at (6,0)  {\fontsize{7}{5}\selectfont$  \,\, [0,2,0]$};
\node[above right] at (8,0)  {\fontsize{7}{5}\selectfont$\,\,  [0,2,1]$};
\node[above right] at (10,0)  {\fontsize{7}{5}\selectfont$\,\,  [0,2,2]$};
\node[above right] at (12,0)  {\fontsize{7}{5}\selectfont$\,\,  [0,2,3]$};

\node[above right] at (1,-1.75)  {\fontsize{7}{5}\selectfont$\,\,  [0,1,-3]$};
\node[above right] at (3,-1.75)  {\fontsize{7}{5}\selectfont$\,\,  [0,1,-2]$};
\node[above right] at (5,-1.75)  {\fontsize{7}{5}\selectfont$\,\,  [0,1,-1]$};
\node[above right] at (7,-1.75)  {\fontsize{7}{5}\selectfont$\,\,  [0,1,0]$};
\node[above right] at (9,-1.75)  {\fontsize{7}{5}\selectfont$\,\,  [0,1,1]$};
\node[above right] at (11,-1.75)  {\fontsize{7}{5}\selectfont$\,\,  [0,1,2]$};

\node[above right] at (0,-3.5)  {\fontsize{7}{5}\selectfont$ \,\, [0,0,-4]$};
\node[above right] at (2,-3.5)  {\fontsize{7}{5}\selectfont$ \,\,  [0,0,-3]$};
\node[above right] at (4,-3.5)  {\fontsize{7}{5}\selectfont$ \,\,  [0,0,-2]$};
\node[above right] at (6,-3.5)  {\fontsize{7}{5}\selectfont$  \,\, [0,0,-1]$};
\node[above right] at (8,-3.5)  {\fontsize{7}{5}\selectfont$\,\,  [0,0,0]$};
\node[above right] at (10,-3.5)  {\fontsize{7}{5}\selectfont$\,\,  [0,0,1]$};
\node[above right] at (12,-3.5)  {\fontsize{7}{5}\selectfont$\,\,  [0,0,2]$};

\end{tikzpicture}

\end{figure}

Since the apartment $\mathcal{A}$ corresponds to the frame generated by the standard basis, we obtain a transitive action on $\mathcal{A}$ by multiplying the homothety classes corresponding to its vertices on the left by monomial matrices. We state the following easy to check facts: diagonal matrices of the form $(\bm{\pi}^{\beta_i}\delta_{ij})$ where $\delta_{ij}$ is the Kronecker delta, act on the apartment by translations, row-interchanging elementary matrices act by  reflections with respect to hyperplanes passing through the vertex $[0,0,\dots,0]$ (which we will also refer to as the origin), and  monomial matrices of the form $(\bm{\pi}^{\beta_i}\delta_{\sigma(i)j}) , \sigma\in S_n$  correspond to compositions of such reflections and translations. The type of such a matrix can give us information about the action of the matrix on the apartment.

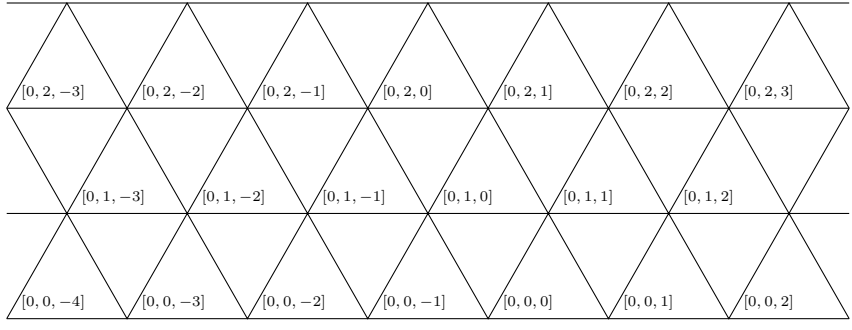
\begin{lem}
\label{type0}
A monomial matrix of the form $(\bm{\pi}^{\beta_i}\delta_{\sigma(i)j})$ has type $0$ if and only if it acts as a product of reflections on  $\mathcal{A}$.
\end{lem}

\begin{proof} Let $\xi \colonequals ({\pi}^{\beta_i}\delta_{\sigma(i)j})$, then the action of $\xi$ preserves the frame $\{e_1, \dots, e_n\}$ and therefore acts on the apartment. Since $t(g) \equiv 0 \pmod{n}$, either $\xi \in \SL_n(D)$ or we can write it as $\xi=u\cdot \xi'$, where $u$ is a row-interchanging matrix and $\xi'\in \SL_n(D)$.  By the obervation above, $u$ acts on $\mathcal{A}$ by a reflection, so it is enough to show that $\xi \in \SL_n(D)$ acts on the buildings as a product of reflections. This is indeed true; for example, see Section 8 in \cite{brown} for the construction of the associated Weyl group.

Conversely, suppose we have a monomial matrix corresponding to the reflection with respect to the hyperplane $x_s-x_t=\mu$. Then we can easily check that  $\eta=(\bm{\pi}^{\alpha_i}\delta_{\tau(i)j})$, 
where $\tau=(st)$,  
\[
\alpha_i = \begin{cases} 
      0 &  i \ne s \text{ and }  i\ne t\\
      \mu & i=s\\
      -\mu& i=t
   \end{cases}
\] and  $t(\eta)\equiv 0 \pmod{n}$. Any product of such reflections will also be a monomial matrix of type $0$ of the form $(\bm{\pi}^{\beta_i}\delta_{\sigma(i)j})$.
\end{proof}

We can extend the concept of type to lattices as follows. The group $\text{GL}_n(D)$ acts transitively on the set of $\Delta$-lattices, the action being well-defined on homothety classes.  Setting $L_0$ as above to have type $0$, the type of any other lattice $L\coloneqq gL_0$ is defined as 
\[
t(L) \equiv t(g) \pmod{n},
\] where $t(g)$ is the type of the matrix defined earlier. Note that the type of a lattice is well defined on each homothety class. It follows then that $g \in \GL_n(D)$ will send a vertex $[L]$ to a vertex $[L']$ with 
\begin{equation}
\label{def types}
t(L') \equiv t(L)+t(g) \pmod{n}.
\end{equation}
Since $[L]=[m_1, m_2, \dots, m_n]$  is obtained from $L_0$ by (left) multiplication with the diagonal matrix $(\bm{\pi}^{m_i}\delta_{ij})$,  we have $$t(L)\equiv\sum_{i=1}^n m_i \pmod{n}.$$

There exists an immediate connection between the building for $\SL_n(D)$ and maximal orders in $M_n(D)$, which we explore next.

%
%
%

\subsection{Maximal orders.}

We  introduce  the correspondence between  maximal orders in $M_n(D)$ and  homothety classes of lattices.  Following Chapter 5 in \cite{reiner}, let  $L_0$ be the free (right) $\Delta$-lattice generated by $\{e_1, \dots, e_n\}$, then we identify $\text{End}_{\Delta}(L_0)$ with the maximal order $M_n(\Delta)$. By \cite[(17.3)]{reiner},  every maximal order in $M_n(D)$ is conjugate to $M_n(\Delta)$ by some $\xi \in \text{GL}_n(D)$, so we identify $\xi M_n(\Delta)\xi^{-1}$ with $\text{End}_{\Delta}(\xi L_0)$.  One can easily check that $[L_1]=[L_2]$ if and only if $\text{End}_{\Delta}(L_1)=\text{End}_{\Delta}(L_2)$, so we can identify each homothety class of a lattice with a maximal order.  By the previous subsection, we have correspondences between homothety classes of lattices, maximal orders, and vertices in the building.

Next, we look at normalizers of maximal orders. Let $\Lambda=M_n(\Delta)$, in which case  $\Lambda^\times = \GL_n(\Delta)$.
\begin{prop}
\label{maxnorm} For $\Lambda=M_n(\Delta)$, the normalizer is given by  $\mathcal{N}(\Lambda)=D^\times \Lambda^\times$.
\end{prop}

\begin{proof}
 We follow a similar approach to the discussion in \cite[Section 3.1]{local}.  Since $\Delta$ is the unique maximal order in $D$, note that $x \Delta x^{-1}=\Delta$ for any $x \in D^\times$. Embedding $D^\times \hookrightarrow \GL_n(D)$ diagonally,  $xM_n(\Delta)x^{-1}=M_n(x\Delta x^{-1})=M_n(\Delta)$, so $D^\times  \subseteq \mathcal{N}(\Lambda)$. In addition, clearly $\Lambda^\times \subseteq \mathcal{N}(\Lambda)$, and therefore $D^\times \Lambda^\times \subseteq \mathcal{N}(\Lambda)$.

Now we prove the other containment.  From (37.25)-(37.27) of \cite{reiner}, 
\begin{equation*}
\mathcal{N}(\Lambda)/k^\times \Lambda^\times \cong \Z/m\Z.
\end{equation*}
By \cite[17.3]{reiner}, $\bm{\pi}\Lambda$ is the unique two-sided ideal of $\Lambda$, which implies $\bm{\pi} \in \mathcal{N}(\Lambda)$. Since $m$ is the smallest power such that $\bm{\pi}^m\in k^\times$, it follows that the normalizer is generated by the set $\{ \bm{\pi}, k^\times \Lambda^\times \}$. We already have $D^\times\Lambda^\times \subseteq \mathcal{N}(\Lambda)$, so $D^\times\Lambda^\times \subseteq \langle \bm{\pi}, k^\times\Lambda^\times \rangle$. On the other hand, $\langle \bm{\pi}, k^\times \Lambda^\times \rangle \subseteq D^\times\Lambda^\times$ since $\bm{\pi} \in D^\times$. Thus $\mathcal{N}(\Lambda)=D^\times \Lambda^\times$.
\end{proof}

 Since all maximal orders in $M_n(D)$ are conjugate to $M_n(\Delta)$, we have the following easy corollary:

\begin{cor}
\label{max order norm}
Suppose $\Lambda$ is a maximal order in $A= M_n(D)$. Then $$\nr_{A/k}(\mathcal{N}(\Lambda))=(k^\times)^n R^\times.$$
\end{cor}

\begin{proof}
Let $\Lambda$ be a maximal order in $M_n(D)$.Then $\Lambda=\xi M_n(\Delta)\xi^{-1}$ for some $\xi \in \GL_n(D)$.  By Proposition \ref{maxnorm}, 

\[
\mathcal{N}(\Lambda) = \mathcal{N}( \xi M_{n}(\Delta) \xi^{-1})=\xi \mathcal{N}(M_n(\Delta))\xi^{-1}=\xi D^\times \GL_n(\Delta)\xi^{-1}.\]

 Since norms are multiplicative and $k$ is commutative, we have $$\nr_{A/k}(\xi D^\times \GL_n(\Delta)\xi^{-1})=\nr_{A/k}(D^\times)\nr_{A/k}(\GL_n(\Delta)).$$ 
 By Equation (\ref{nrD}) and the fact that $\nr_{D/k}(D)=k$ (see page 153 in \cite{reiner}), we get $\nr_{A/k}(D^\times)=(k^\times)^n$. That $\nr_{A/k}(\GL_n(\Delta))=R^\times$ follows from  Equation (\ref{deltanrd}).\end{proof}

\subsection{Tiled orders.}
\label{tiled}

We  use the building-theoretic framework to study the following  orders.

\begin{defn}
 We say  $\Gamma \subseteq M_n(D)$ is a tiled order if it contains a conjugate of the diagonal ring $\text{diag}(\Delta, \Delta, \dots, \Delta)$.
\end{defn}

For example, maximal orders are tiled, since every maximal order in $M_n(D)$ is conjugate to $M_n(\Delta)$. When $n=1$, $\Delta$ is the unique maximal order in $D$ and therefore equal to any of its conjugates, so  $\Gamma \subset D$ is tiled if and only if it is maximal and $\Gamma=\Delta$.

None of the results in this subsection are new, but we include them together with examples for the convenience of the reader. We describe the connection  between tiled orders and buildings following \cite{poly}, while the more specific results involving structural invariants follow \cite{ants}.

Note that conjugation does not change the structure of the normalizer of an order, and simply might change the apartment we work in. Conjugating if necessary, from  now on  we may and will assume that $\Gamma$  contains $\text{diag}(\Delta, \Delta, \dots, \Delta)$, and we fix the apartment $\mathcal{A}$ where $[0,0, \dots, 0]$ corresponds to $M_n(\Delta)$. Then by Proposition 2.1 in \cite{poly}, $\Gamma=(\p_\nu^{\mu_{ij}})$ where  $\p=\bm{\pi}\Delta$ and 
\begin{equation*}
\mu_{ij}+\mu_{jk}\ge \mu_{ik} \quad \text{for all} \, i,j,k \le n, \quad \mu_{ii}=0.
\end{equation*}

We denote by $M_\Gamma=(\mu_{ij})$ the \textit{exponent matrix} of $\Gamma$.   We associate to $\Gamma$ a polytope  in the apartment the following way. The equations of the form $x_i-x_j=\mu \in \Z, \, 1\le i,j \le n$ determine hyperplanes in $\R^{n-1}$, and the hyperplanes $H_{ij}\coloneqq x_i-x_j=\mu_{ij}$ with $\mu_{ij}$  as above are the bounding hyperplanes of a convex polytope, which we denote by $C_\Gamma$. In addition, the vertices given by $$[P_1]=[\mu_{11}, \mu_{21},\dots, \mu_{n1}],\,\, [P_2]=[\mu_{12}, \mu_{22},\dots, \mu_{n2}], \,\, \dots, \,\, [P_n]=[\mu_{1n}, \mu_{2n},\dots, \mu_{nn}],$$ which correspond to the homothety classes of the  columns of  $\Gamma$,  are extremal points on $C_\Gamma$  \cite[Proposition 2.2]{normalizers}  and  uniquely determine $\Gamma$ \cite[Remark II.4]{graduated}. From now on, we will refer to the homothety classes $[P_i]$ as the \textit{distinguished vertices} of $C_\Gamma$.

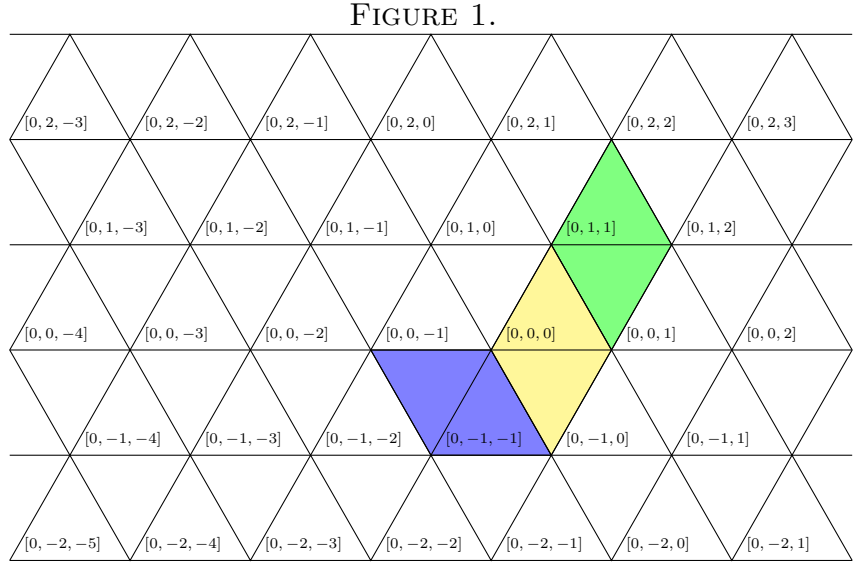
\begin{figure}[h]
\caption{}
\label{general}

\smallskip

\begin{tikzpicture}[scale=0.8, every node/.style={transform shape}]

\draw[fill=blue!50] (8,-3.5)--(9,-5.25)--(7,-5.25) -- (6,-3.5) -- (8,-3.5);
\draw[fill=green!50]  (10,-3.5) -- (9,-1.75) -- (10, 0) -- (11, -1.75) -- (10, -3.5);
\draw[fill=yellow!50]  (10,-3.5) -- (9,-1.75) -- (8, -3.5) -- (9, -5.25) -- (10, -3.5);

\draw (0,0) -- +(1,1.75) -- (2,0) -- +(1,1.75)-- (4,0) -- +(1,1.75)-- (6,0) -- +(1,1.75) -- (8,0) -- +(1,1.75) -- (10,0) -- +(1,1.75) -- (12,0) -- +(1,1.75) -- (14,0);

\draw (0,1.75) -- (14, 1.75);

\draw (0,0) -- (14, 0);

\draw (0,0) -- (1,-1.75) -- (2,0) -- (3,-1.75)-- (4,0) -- (5,-1.75)-- (6,0) -- (7,-1.75) -- (8,0) -- (9,-1.75) -- (10,0) -- (11,-1.75) -- (12,0) -- (13,-1.75) -- (14,0);

\draw (0,-1.75) -- (14, -1.75);

\draw (0,-3.5) -- +(1,1.75) -- (2,-3.5) -- +(1,1.75)-- (4,-3.5) -- +(1,1.75)-- (6,-3.5) -- +(1,1.75) -- (8,-3.5) -- +(1,1.75) -- (10,-3.5) -- +(1,1.75) -- (12,-3.5) -- +(1,1.75) -- (14,-3.5);

\draw (0,-3.5) -- (14, -3.5);

\draw (0,-3.5) -- (1,-5.25) -- (2,-3.5) -- (3,-5.25)-- (4,-3.5) -- (5,-5.25)-- (6,-3.5) -- (7,-5.25) -- (8,-3.5) -- (9,-5.25) -- (10,-3.5) -- (11,-5.25) -- (12,-3.5) -- (13,-5.25) -- (14,-3.5);

\draw (0,-5.25) -- (14, -5.25);

\draw (0,-7) -- +(1,1.75) -- (2,-7) -- +(1,1.75)-- (4,-7) -- +(1,1.75)-- (6,-7) -- +(1,1.75) -- (8,-7) -- +(1,1.75) -- (10,-7) -- +(1,1.75) -- (12,-7) -- +(1,1.75) -- (14,-7);

\draw (0,-7) -- (14, -7);

\node[above right] at (0,0)  {\fontsize{7}{5}\selectfont$ \,\, [0,2,-3]$};
\node[above right] at (2,0)  {\fontsize{7}{5}\selectfont$ \,\,  [0,2,-2]$};
\node[above right] at (4,0)  {\fontsize{7}{5}\selectfont$ \,\,  [0,2,-1]$};
\node[above right] at (6,0)  {\fontsize{7}{5}\selectfont$  \,\, [0,2,0]$};
\node[above right] at (8,0)  {\fontsize{7}{5}\selectfont$\,\,  [0,2,1]$};
\node[above right] at (10,0)  {\fontsize{7}{5}\selectfont$\,\,  [0,2,2]$};
\node[above right] at (12,0)  {\fontsize{7}{5}\selectfont$\,\,  [0,2,3]$};

\node[above right] at (1,-1.75)  {\fontsize{7}{5}\selectfont$\,\,  [0,1,-3]$};
\node[above right] at (3,-1.75)  {\fontsize{7}{5}\selectfont$\,\,  [0,1,-2]$};
\node[above right] at (5,-1.75)  {\fontsize{7}{5}\selectfont$\,\,  [0,1,-1]$};
\node[above right] at (7,-1.75)  {\fontsize{7}{5}\selectfont$\,\,  [0,1,0]$};
\node[above right] at (9,-1.75)  {\fontsize{7}{5}\selectfont$\,\,  [0,1,1]$};
\node[above right] at (11,-1.75)  {\fontsize{7}{5}\selectfont$\,\,  [0,1,2]$};

\node[above right] at (0,-3.5)  {\fontsize{7}{5}\selectfont$ \,\, [0,0,-4]$};
\node[above right] at (2,-3.5)  {\fontsize{7}{5}\selectfont$ \,\,  [0,0,-3]$};
\node[above right] at (4,-3.5)  {\fontsize{7}{5}\selectfont$ \,\,  [0,0,-2]$};
\node[above right] at (6,-3.5)  {\fontsize{7}{5}\selectfont$  \,\, [0,0,-1]$};
\node[above right] at (8,-3.5)  {\fontsize{7}{5}\selectfont$\,\,  [0,0,0]$};
\node[above right] at (10,-3.5)  {\fontsize{7}{5}\selectfont$\,\,  [0,0,1]$};
\node[above right] at (12,-3.5)  {\fontsize{7}{5}\selectfont$\,\,  [0,0,2]$};

\node[above right] at (1,-5.25)  {\fontsize{7}{5}\selectfont$\,\,  [0,-1,-4]$};
\node[above right] at (3,-5.25)  {\fontsize{7}{5}\selectfont$\,\,  [0,-1,-3]$};
\node[above right] at (5,-5.25)  {\fontsize{7}{5}\selectfont$\,\,  [0,-1,-2]$};
\node[above right] at (7,-5.25)  {\fontsize{7}{5}\selectfont$\,\,  [0,-1,-1]$};
\node[above right] at (9,-5.25)  {\fontsize{7}{5}\selectfont$\,\,  [0,-1,0]$};
\node[above right] at (11,-5.25)  {\fontsize{7}{5}\selectfont$\,\,  [0,-1,1]$};

\node[above right] at (0,-7)  {\fontsize{7}{5}\selectfont$ \,\, [0,-2,-5]$};
\node[above right] at (2,-7)  {\fontsize{7}{5}\selectfont$ \,\,  [0,-2,-4]$};
\node[above right] at (4,-7)  {\fontsize{7}{5}\selectfont$ \,\,  [0,-2,-3]$};
\node[above right] at (6,-7)  {\fontsize{7}{5}\selectfont$  \,\, [0,-2,-2]$};
\node[above right] at (8,-7)  {\fontsize{7}{5}\selectfont$\,\,  [0,-2,-1]$};
\node[above right] at (10,-7)  {\fontsize{7}{5}\selectfont$\,\,  [0,-2,0]$};
\node[above right] at (12,-7)  {\fontsize{7}{5}\selectfont$\,\,  [0,-2,1]$};

\end{tikzpicture}

\end{figure}

\medskip

\begin{example}
\label{ex1} Let $\Gamma$ with exponent matrix  $M_\Gamma=\left( \begin{array}{ccc}
0 & 1 &1 \\
0 & 0 & 1\\
0 & 1 & 0\end{array} \right)$.
In Figure \ref{general} we see the associated convex polytope $C_\Gamma$ in blue determined by
$$0 \le x_1-x_2\le 1$$
$$0 \le x_1-x_3\le 1$$
$$-1 \le x_2-x_3\le 1.$$ Note that $C_\Gamma$ is also the convex hull of its distinguished vertices $$[P_1]=[0,0,0], \quad  [P_2]=[0,-1,0] = [1,0,1], \quad [P_3]=[0,0,-1]=[1,1,0].$$

 We also see that the action of $\GL_n$ on the apartment extends to polytopes and their associated orders. For example, consider the matrix $\xi=\left( \begin{array}{ccc}
0 & 0 &1 \\
0 & \bm{\pi} & 0\\
\bm{\pi} & 0 & 0\end{array} \right)$, then $\xi \Gamma \xi^{-1}=\Gamma'$ with exponent matrix  $M_{\Gamma'}=\left( \begin{array}{ccc}
0 & 0 &-1 \\
2 & 0 & 0\\
2 & 1 & 0\end{array} \right)$ and convex polytope $C_{\Gamma'}$ depicted in Figure \ref{general} in green. Notice that $\xi=\left( \begin{array}{ccc}
1 & 0 &0 \\
0 & \bm{\pi} & 0\\
0& 0 & \bm{\pi}\end{array} \right)\left( \begin{array}{ccc}
0 & 0 &1 \\
0 & 1& 0\\
1 & 0 & 0\end{array} \right)$ decomposes as a product of a diagonal and a permutation matrix. The permutation matrix corresponds to the reflection with respect to the hyperplane $x_1-x_3=0$, which reflects $C_\Gamma$ to the polytope $C_{\widetilde{\Gamma}}$ in yellow. The diagonal matrix then corresponds to translating $C_{\widetilde{\Gamma}}$ so that the vertex $[0,1,1]$ aligns with $[0,2,2]$, which then gives $C_{\Gamma'}$. 

\end{example}

The first connection between the geometry of $C_\Gamma$ and algebraic properties of $\Gamma$ is the following:

\begin{prop}[Proposition 3.3, \cite{poly}]
Let $\Gamma$ be a tiled order with convex polytope $C_\Gamma$, and consider the set of maximal orders $\{\Lambda_i\}$ corresponding to the vertices on $C_\Gamma$. Then $\Gamma=\cap_i \Lambda_i$.
\end{prop}


We can make the connection between the geometry of $C_\Gamma$ and the algebraic properties of $\Gamma$ more explicit.  In unpublished work \cite{zas} (see \cite{graduated}), Zassenhaus  introduced a set of \textit{structural invariants}  for  tiled orders, defined by:\[ m_{ij\ell}=\mu_{ij}+\mu_{j\ell}-\mu_{i\ell}, \, \text{for } 1\le i,j,\ell\le n.\] Each isomorphism class is then determined by these invariants, and we have the following isomorphism criterion.  

\begin{prop}[Zassenhaus, \cite{zas}]
\label{isom}
Let $\Gamma, \Gamma' \subset M_n(D)$ be two tiled orders containing \\
$\text{diag}(\Delta,\Delta, \dots, \Delta)$, and let $m_{ij\ell}$ and  $m'_{ij\ell}$  be their structural invariants. Then $\Gamma$ and $\Gamma'$ are isomorphic if and only if there exists $\sigma \in S_n$ such that \[m'_{ij\ell}=m_{\sigma(i)\sigma(j)\sigma(\ell)}\quad \text{ for all } \quad 1\le i,j,\ell\le n.\]
\end{prop}

\begin{proof} See the proof in  \cite[Proposition 6]{ants} generalizing  to matrices over a division ring.
\end{proof}

Denote by $t_i=t[P_i]$ the types of the distinguished vertices of $\Gamma$. 

\begin{cor}
\label{movep}
Let $\Gamma$ be a tiled order with structural invariants $m_{ij\ell}$ and types of distinguished vertices $t_i$, and let $\xi=(\bm{\pi}^{\alpha_i}\delta_{\sigma(i)j})$ be a monomial matrix. Then $\Gamma' \coloneqq \xi \Gamma \xi^{-1}$ has structural invariants and types  of distinguished vertices 
\[ m'_{ij\ell}=m_{\sigma(i)\sigma(j)\sigma(\ell)} \quad \text{and} \quad t'_i=t(\xi)+t_{\sigma(i)}\] for all $i,j,\ell \le n$, and $\xi [P_{\sigma(i)}]=[P'_i]$ for all $i \le n$.
\end{cor}

\begin{proof}
Let $\Gamma=(\p^{\mu_{ij}})$. A direct calculation shows that $\xi \Gamma \xi^{-1}=(\p^{\alpha_i-\alpha_j+\mu_{\sigma(i)\sigma(j)}})$ and the claim about the structural invariants  follows easily. 

Let $\Lambda_{\sigma(\ell)}$ and $\Lambda'_{\ell}$ be the maximal orders corresponding to the distinguished vertices  $[P_{\sigma(\ell)}]=[\mu_{i\sigma(\ell)}]$ and $[P'_\ell]=[\alpha_i-\alpha_\ell+\mu_{\sigma(i)\sigma(\ell)}]$. By \cite[Corollary 2.3]{poly}, $\Lambda_{\sigma(\ell)}=(\p^{\mu_{i\sigma(\ell)}-\mu_{j\sigma(\ell)}})$ and  $\Lambda'_{\ell}=(\p^{\alpha_i-\alpha_j+\mu_{\sigma(i)\sigma(\ell)}-\mu_{\sigma(j)\sigma(\ell)}})$. Another direct calculation shows that $\xi \Lambda_{\sigma(\ell)}\xi^{-1}=(\p^{\alpha_i-\alpha_j+\mu_{\sigma(i)\sigma(\ell)}-\mu_{\sigma(j)\sigma(\ell)}})=\Lambda'_\ell$.
The claim about the types follows from Equation (\ref{def types}).
\end{proof}

As seen in \cite[Proposition 5]{ants}, the structural invariants completely determine the geometry of $C_\Gamma$. In particular, if for two tiled orders $\Gamma$ and $\Gamma'$ we have $m_{ijk}=m'_{\sigma(i)\sigma(j)\sigma(k)}$, we will say that $C_\Gamma$ and $C_{\Gamma'}$ are \textit{congruent}. For example, all three polytopes in Figure \ref{general} are congruent. One particular case is when the structural invariants of two polytopes agree.

\begin{cor}
\label{translation}
Let $\Gamma, \Gamma' \subset M_n(D)$ be two tiled orders containing 
$\text{diag}(\Delta, \Delta, \dots, \Delta)$, and let $m_{ij\ell}$ and  $m'_{ij\ell}$ be their structural invariants. If  $m'_{ij\ell}=m_{ij\ell}$ for all $1\le i,j,\ell\le n$, then $C_\Gamma$ is a translation of $C_{\Gamma'}$.
\end{cor}

\begin{proof}
In the proof of Proposition \ref{isom}, we obtain the matrix $\xi=(\bm{\pi}^{\alpha_i}\delta_{\sigma(i)j})$ where  $\xi\Gamma\xi^{-1}=\Gamma'$ and $m_{ij\ell}=m'_{\sigma(i)\sigma(j)\sigma(\ell)}$ for all $i,j,\ell \le n$. If $m'_{ij\ell}=m_{ij\ell}$, we take $\sigma$ to be the identity permutation, and the resulting matrix $\xi$ is a diagonal matrix. Since diagonal matrices act on the apartment by translations,  $\xi$ will translate $C_\Gamma$ to $C_{\Gamma'}$.
\end{proof}

\begin{example} Let $\Gamma$ and $\Gamma'$ be the tiled orders in Example \ref{ex1}, where $\Gamma$ has  structural invariants $$(m_{123}, m_{132}, m_{213}, m_{231}, m_{312}, m_{321})=(1, 1, 0, 1, 0, 1),$$ and $\Gamma'$ has structural invariants $$(m'_{123}, m'_{132}, m'_{213}, m'_{231}, m'_{312}, m'_{321})=(1, 0, 1, 0, 1, 1).$$  Then $m'_{ij\ell}=m_{\sigma(i)\sigma(j)\sigma(\ell)}$ for $\sigma=(13)$, so the two orders are isomorphic. Indeed, we notice that $C_\Gamma$ depicted in Figure \ref{general} in blue has the same shape (and size) as $C_{\Gamma'}$ depicted in green, and the matrix $\xi$ in Example \ref{ex1} gives the desired isomorphism. 
\end{example}

Structural invariants, by encoding isomorphism classes of tiled orders, also encode information about their normalizer.

\begin{prop}[Proposition 7,  \cite{ants}]
\label{thmants}
Let $\Gamma=(\p^{\mu_{ij}})$ be a tiled order and $\{m_{ij\ell} \, | \, i,j, \ell \le n\}$ its set of structural invariants. Then 
\begin{equation*}
\mathcal{N}(\Gamma)=\bigcup_{\xi_\sigma \in \widetilde{H}}\xi_{\sigma}D^\times\Gamma^\times,
\end{equation*} where $\widetilde{H}\coloneqq \{\xi_\sigma=(\bm{\pi}^{\mu_{i1}-\mu_{\sigma(i)\sigma(1)}}\delta_{\sigma(i)j}) \, | \, \sigma \in H\}$ and $H$ is the subgroup of $S_n$ given by $H=\{\sigma \in S_n \, | \, m_{ij\ell}=m_{\sigma(i)\sigma(j)\sigma(\ell)} \,\, \text{for all} \,\, i,j,\ell \le n\}$.

\end{prop}

Geometrically,  by Corollary \ref{movep} the elements $\xi_\sigma$ in Proposition \ref{thmants} permute the distinguished vertices and therefore give a symmetry of $C_\Gamma$.


%

\section{Tiled orders and the norm of their normalizers}
\label{tiled loc}

Our goal is to compute $\nr(\mathcal{N}(\Gamma))$ for any given tiled order $\Gamma$. Proposition \ref{thmants} already gives us a naive  algorithm for finding $\mathcal{N}(\Gamma)$, and the algorithm in \cite{ants} allows us to speed up the process in some cases, although  the algorithm could still take factorial time for some   tiled orders.  In this section, we investigate $\nr(\mathcal{N}(\Gamma))$ using the building-theoretic framework, which allows us to bypass the task of finding the normalizer when the degree $n$ of the algebra $M_n(D)$ is prime. However, as seen in Algorithm \ref{alg1}, the composite case  still involves finding $\mathcal{N}(\Gamma)$.

\subsection{Algebraic considerations}

We start by describing $\nr(\mathcal{N}(\Gamma))$ algebraically.  Let $\Gamma=(\p^{\mu_{ij}})$ be a tiled order in $A=M_n(D)$ with associated convex polytope $C_\Gamma$.   Denote the structural invariants of $\Gamma$ by $m_{ij\ell}$, let $\widetilde{H}\coloneqq \{\xi_\sigma=(\bm{\pi}^{\mu_{i1}-\mu_{\sigma(i)\sigma(1)}}\delta_{\sigma(i)j}) \, | \, \sigma \in H\}$ where  $H=\{\sigma \in S_n \, | \, m_{ij\ell}=m_{\sigma(i)\sigma(j)\sigma(\ell)} \,\, \text{for all} \,\, i,j,\ell \le n\}$.
   Proposition \ref{thmants} gives that $\displaystyle \mathcal{N}(\Gamma)=\bigcup_{\xi_\sigma \in \widetilde{H}}\xi_{\sigma}D^\times\Gamma^\times$, so 
\[
\nr_{A/k}(\mathcal{N}(\Gamma))=\bigcup_{\xi_\sigma \in \widetilde{H}}\nr_{A/k}(\xi_\sigma D^\times\Gamma^\times)=\bigcup_{\xi_\sigma \in \widetilde{H}} \nr_{A/k}(\xi_\sigma)\nr_{A/k}(D^{\times})\nr_{A/k}(\Gamma^\times).\]

From now on, we supress the subscript under the reduced norm and denote $\nr = \nr_{A/k}$ when the context is obvious.

\begin{prop}
\label{typesnrd}
With the notation above,  $\nr(\mathcal{N}(\Gamma))=(k^\times)^dR^\times$, where $d\Z/n\Z = \{t(\xi_\sigma): \xi_\sigma \in \widetilde{H}\}$.
\end{prop}

\begin{proof}

 By equation (\ref{nrD}) and the fact that $\nr_{D/k}(D)=k$ (see page 153 in \cite{reiner}), we get $\nr_{A/k}(D^\times)=(k^\times)^n$. Since $\Gamma$ contains the ring $diag(\Delta, 1, 1, \dots,1)$ and is moreover an intersection of maximal orders, each of them conjugate to $M_n(\Delta)$,  by Equation (\ref{deltanrd}) we get   $\nr(\Gamma^\times)=R^\times$.

Therefore 
\begin{equation}
\label{unioneqn}
\nr (\mathcal{N}(\Gamma))=\bigcup_{\xi_\sigma \in \widetilde{H}} \nr (\xi_\sigma)(k^\times)^nR^\times.
\end{equation}

Since the identity permutation $e \in S_n$ gives $\xi_{e}=I_n$ the $n \times n$ identity matrix,  $(k^\times)^nR^\times \subseteq \mathcal{N}(\Gamma)$ and we have nested subgroups $$(k^\times)^nR^\times \subseteq \nr (\mathcal{N}(\Gamma)) \subseteq k^\times.$$  

The valuation $v$ on the field $k$  induces a homomorphism 
\begin{align*}
\hspace{3cm} \nr(\mathcal{N}(\Gamma)) & \rightarrow  \Z/n\Z& \\
  x & \mapsto  v(x) \pmod{n}
\end{align*}
 with kernel $(k^\times)^nR^\times$. Therefore,  $\nr(\mathcal{N}(\Gamma))/(k^\times)^nR^\times$ has the structure of a subgroup of $\Z/n\Z$. By Equation (\ref{unioneqn}), Lemma \ref{teqn}, and the definition of the type of a matrix, we have  $\nr(\mathcal{N}(\Gamma))/(k^\times)^nR^\times =  \langle t(\xi_\sigma): \xi_\sigma \in \widetilde{H} \rangle  \subset \Z/n\Z$.  All we have left to show is that the set  $\{ t(\xi_\sigma): \xi_\sigma \in \widetilde{H} \}$ forms a subgroup of $\Z/n\Z$.

 By Corollary (\ref{movep}), $\xi_\sigma \cdot [P_{\sigma(i)}]=[P_i]$ for all $i \le n$, so $t(\xi_\sigma)\equiv t(P_i)-t(P_{\sigma(i)}) \pmod{n}$ for all $i \le n$.   In particular, $t(\xi_\sigma)\equiv t(P_1)-t(P_{\sigma(1)}) \pmod{n}$. But then $t(\xi_\tau) \equiv t(P_{\sigma(1)})-t(P_{\tau\sigma(1)}) \pmod{n}$ for any other $\xi_\tau \in \widetilde{H}$. Thus
\begin{equation}
\label{types additive}
t(\xi_{\tau\sigma})\equiv t(P_1)-t(P_{\tau \sigma(1)}) \equiv t(\xi_\tau)+t(\xi_\sigma) \pmod{n} \quad \text{for all} \quad \xi_\sigma, \xi_\tau, \xi_{\sigma\tau} \in \widetilde{H}.
\end{equation}

Since $H$ is a subgroup of $S_n$, $\{t(\xi_\sigma): \xi_\sigma \in \widetilde{H}\}$ is closed under addition, contains the identity $t(\xi_e)\equiv 0 \pmod{n}$ and has inverses $-t(\xi_\sigma)\equiv t(\xi_{\sigma^{-1}}) \pmod{n}$.  
\end{proof}

To determine the exponent $d$ above, it is enough to find those $\xi_\sigma \in \widetilde{H}$ with $t(\xi_\sigma) \not\equiv 0 \pmod{n}$. In particular, we don't need to consider the following  subset of $\mathcal{N}(\Gamma)$.

\begin{lem}
\label{disjoint cycles}
Let $\Gamma \subseteq M_n(D)$ be a tiled order, and $\xi_\sigma \in \widetilde{H}$ have associated permutation $\sigma \in H$. Let $\sigma=\sigma_1\sigma_2\dots\sigma_s$ be a decomposition into disjoint cycles of length $l_1, l_2, \dots, l_s$. If any of the cycles $\sigma_i$ has length $l_i$ with $\gcd(l_i,n)=1$, then $t(\xi_\sigma) \equiv 0 \pmod{n}$.
\end{lem}

\begin{proof}
Note that if any of the $l_i=1$,  by Corollary \ref{movep} $\xi_\sigma$ fixes some vertex $[P_j]$ and therefore $t(\xi_\sigma)\equiv 0 \mod{n}$. Thus, we can assume $\sigma$ does not fix any $j \le n$. Without loss of generality, suppose $\gcd(l_1, n)=1$, and let $i \le n$ not fixed by $\sigma_1$. Then  $\sigma^{l_1}=(\sigma_2\sigma_3\dots\sigma_s)^{l_1}$ fixes $i$,  so $t( \xi_{\sigma^{l_1}}) \equiv 0 \pmod{n}$. By Equation (\ref{types additive}), $t( \xi_{\sigma^{l_1}}) \equiv l_1t(\xi_\sigma) \pmod{n}$, and since $\gcd(l_1,n)=1$, we get $t(\xi_\sigma)\equiv  0 \pmod{n}$.
\end{proof}

\subsection{Geometric interpretations.} We want to describe the factor $d$ from Proposition \ref{typesnrd}  in a building-theoretic way. We proceed by defining an equivalence relation between polytopes congruent to $C_\Gamma$ and show in Theorem \ref{n classes} that $d$ equals the number of equivalence classes in this relation. We offer two interpretations of the equivalence relation: an algebraic and a geometric one.

We set some notation. For any two tiled orders $\Gamma$ and $\Gamma'$, we denote by $m_{ij\ell}, [P_i]$, and respectively, $m_{ij\ell}', [P_i']$, the structural invariants and the distingushed vertices, and by $t_i \coloneqq t(P_i)$ and $t_i' \coloneqq t(P_i')$ the types of the distinguished vertices of $\Gamma$, and respectively, $\Gamma'$.

\begin{defn}
\label{class}

Let $\Gamma$ and $\Gamma'$ be two tiled orders. 
Define the following relation: $\Gamma \sim \Gamma'$ if and only if  there exists $\sigma \in S_n$ such that \begin{equation*}
\begin{array}{cllccll}
m_{ij\ell}'& = & m_{\sigma(i)\sigma(j)\sigma(\ell)}  & \text{and} &  t_i'& \equiv  &t_{\sigma(i)}  \pmod{n}
\end{array}
\end{equation*}
for all $i,j,\ell \le n$.
\end{defn}

The relation just defined is clearly an equivalence relation. When $\Gamma \sim \Gamma'$, we also write $$[\Gamma]=[(m_{ij\ell}), (t_1,t_2, \dots, t_n)] = [(m'_{ij\ell}), (t_1',t_2',\dots, t_n')]=[\Gamma'],$$ where the tuples $(m_{ij\ell})$ and $(m_{ij\ell}')$ are in lexicographical order of the indices $i,j,\ell \le n$.

By Proposition \ref{isom}, two equivalent tiled orders are isomorphic. Recall that $\GL_n(D)$ act on the set of tiled orders isomorphic to $\Gamma$  by conjugation. We get a similar action on the equivalence classes just defined by restricting ourselves to monomial matrices, which preserve the apartment $\mathcal{A}$.

\begin{lem}
\label{action lem}
Let $N \subset \GL_n(D)$ be the subgroup of monomial matrices. Then $N$ acts on the equivalence classes defined above by $\xi[\Gamma]=[\xi \Gamma \xi^{-1}]$.
In particular, if $\xi=(\bm{\pi}^{\alpha_i}\delta_{\tau(i)j}) \in N$ and $[\Gamma]=[(m_{ij\ell}), (t_i)]$, the action gives 
\begin{equation*}
\begin{array}{cll}
\xi [\Gamma]& = & [(m_{\tau(i)\tau(j)\tau(\ell)}), (t(\xi)+t_{\tau(1)}, \dots, t(\xi)+t_{\tau(n)})]
\end{array}
\end{equation*}
\end{lem}
\begin{proof} 
Since $\xi$ conjugates $\Gamma$, it satisfies the rules of group actions. We need to show the action is well-defined.

Let $\Gamma, \Gamma'$ be tiled orders with structural invariants and types $m_{ij\ell}, t_i$ and $m'_{ij\ell}, t'_i$, such that $[\Gamma]=[\Gamma']$. Let $\xi=(\bm{\pi}^{\alpha_i}\delta_{\tau(i)j}) \in N$ be a monomial matrix, and consider $\widetilde{\Gamma}=\xi \Gamma\xi^{-1}$ and $\widetilde{\Gamma}'=\xi \Gamma'\xi^{-1}$ to be tiled orders with structural invariants  $\widetilde{m}_{ij\ell}, \widetilde{t}_i$ and $\widetilde{m}'_{ij\ell}, \widetilde{t}'_i$. We want to show there exists $\epsilon \in S_n$ such that 
\begin{equation*}
\begin{array}{cllccll}
\widetilde{m}_{ij\ell}'& = & \widetilde{m}_{\epsilon(i)\epsilon(j)\epsilon(\ell)} & \text{and} & \widetilde{t}'_i& \equiv  &\widetilde{t}_{\epsilon(i)} \pmod{n}
\end{array}
\end{equation*}

By Corollary \ref{movep}, given  $\widetilde{\Gamma}=\xi \Gamma\xi^{-1}$ and $\widetilde{\Gamma}'=\xi \Gamma'\xi^{-1}$, we have
\begin{equation*}
\begin{array}{cllccll}
\widetilde{m}_{ij\ell}& = & m_{\tau(i)\tau(j)\tau(\ell)} & \text{and} & \widetilde{t}_i& \equiv  &t(\xi)+t_{\tau(i)} \pmod{n}.\\
\widetilde{m}'_{ij\ell}& = & m'_{\tau(i)\tau(j)\tau(\ell)} & \text{and}  & \widetilde{t}_i'& \equiv  &t(\xi)+t'_{\tau(i)} \pmod{n}.
\end{array}
\end{equation*}
for all $i,j,\ell \le n$.

Since $[\Gamma]=[\Gamma']$, there also exists $\sigma \in S_n$ such that
\begin{equation*}
\begin{array}{cllccll}
m_{ij\ell}'& = & m_{\sigma(i)\sigma(j)\sigma(\ell)}  & \text{and} &  t_i'& \equiv  &t_{\sigma(i)}  \pmod{n}
\end{array}
\end{equation*}
for all $i,j,\ell \le n$. Our claim follows from taking $\epsilon=\tau^{-1}\sigma\tau$.

Therefore $\xi[\Gamma]=[\xi \Gamma^{-1} \xi^{-1}]$, and $[\xi \Gamma \xi^{-1}]=[(m_{\tau(i)\tau(j)\tau(\ell)}), (t(\xi)+t_{\tau(1)}, \dots, t(\xi)+t_{\tau(n)})]$ follows from Corollary \ref{movep}.
\end{proof}

Since equivalent tiled orders are isomorphic, the equivalence relation partitions  convex polytopes congruent to $C_\Gamma$ in the apartment $\mathcal{A}$  into distinct classes. Given that monomial matrices also act on $\mathcal{A}$, we would like a geometric interpretation of these equivalence classes.

\begin{example} 
\label{trivial normalizer}
Let $\Gamma$, $\Gamma'$ and $\Gamma''$ have exponent matrices  \[M_{\Gamma}=\left( \begin{array}{ccc}
0 & 1 &2 \\
0 & 0 & 1\\
0 & 1& 0\end{array} \right) \qquad M_{\Gamma'}=\left( \begin{array}{ccc}
0 & -1 &-1 \\
3 & 0 & 1\\
2 & 1& 0\end{array} \right)  \qquad M_{\Gamma''}=\left( \begin{array}{ccc}
0 & 0 &2 \\
1 & 0 & 2\\
0 & 0& 0\end{array} \right).\]  Since $m_{ijj}=0,$ and $ m_{iji}=m_{ij\ell}+m_{ji\ell}$, we can restrict ourselves to computing the invariants $m_{ij\ell}$ for $ i\ne j \ne \ell \ne i$. Then
\begin{center}
$[\Gamma]=[(m_{123}, m_{132}, m_{213}, m_{231}, m_{312}, m_{321}), (t_1, t_2, t_3)]=[(0,2,1,1,0,1),(0,2,0)]$

$[\Gamma']=[(m_{123}', m_{132}', m_{213}', m_{231}', m_{312}', m_{321}'), (t_1', t_2', t_3')]=[(1,1,1,0,0,2),(2,0,0)].$

$[\Gamma'']=[(m_{123}'', m_{132}'', m_{213}'', m_{231}'', m_{312}'', m_{321}''), (t_1'', t_2'', t_3'')]=[(0,2,1,1,0,1),(1,0,1)].$
\end{center}

Note that for $\sigma=(123)$ we have $m_{ij\ell}'=m_{\sigma(i)\sigma(j)\sigma(\ell)}$ and $t_i'=t_{\sigma(i)}$, so $[\Gamma]= [\Gamma']$. At the same time, there is no $\tau \in S_3$ connecting the structural invariants and types of distinguished vertices of $\Gamma''$ with those of $\Gamma$ or $\Gamma'$, so $[\Gamma]=[\Gamma']\neq [\Gamma'']$. 

In Figure \ref{reflectclass}, $C_\Gamma$ is the polytope in blue, $C_{\Gamma'}$  the polytope in green, and $C_{\Gamma''}$ polytope in red,  all three being congruent. Note that if we reflect $C_\Gamma$ first with respect to $x_1-x_3=0$ and then with respect to $x_1-x_2=-1$, we obtain $C_{\Gamma'}$. We call $C_\Gamma$ and $C_{\Gamma'}$ \textit{reflection equivalent}, since reflections give equivalence relations. However, no product of reflections can send $C_\Gamma$ to  $C_{\Gamma''}$, and they are not reflection equivalent.

\end{example}

We proceed by identifying the link between the equivalence relation defined above and the geometric criterion of reflection equivalence.

\begin{figure}[h]
\caption{}
\label{reflectclass}

\begin{center}

\begin{tikzpicture}[scale=0.8, every node/.style={transform shape}]

\draw[fill=blue!50] (8,-3.5)--(9,-5.25)--(5,-5.25) -- (6,-3.5) -- (8,-3.5);
\draw[fill=green!50] (9,1.75)--(11,-1.75)--(9,-1.75) -- (8,0) -- (9, 1.75);
\draw[fill=purple!50] (7,-1.75)--(8,-3.5)--(4,-3.5) -- (5,-1.75) -- (7, -1.75);

\draw (0,0) -- +(1,1.75) -- (2,0) -- +(1,1.75)-- (4,0) -- +(1,1.75)-- (6,0) -- +(1,1.75) -- (8,0) -- +(1,1.75) -- (10,0) -- +(1,1.75) -- (12,0) -- +(1,1.75) -- (14,0);

\draw (0,1.75) -- (14, 1.75);

\draw (0,0) -- (14, 0);

\draw (0,0) -- (1,-1.75) -- (2,0) -- (3,-1.75)-- (4,0) -- (5,-1.75)-- (6,0) -- (7,-1.75) -- (8,0) -- (9,-1.75) -- (10,0) -- (11,-1.75) -- (12,0) -- (13,-1.75) -- (14,0);

\draw (0,-1.75) -- (14, -1.75);

\draw (0,-3.5) -- +(1,1.75) -- (2,-3.5) -- +(1,1.75)-- (4,-3.5) -- +(1,1.75)-- (6,-3.5) -- +(1,1.75) -- (8,-3.5) -- +(1,1.75) -- (10,-3.5) -- +(1,1.75) -- (12,-3.5) -- +(1,1.75) -- (14,-3.5);

\draw (0,-3.5) -- (14, -3.5);

\draw (0,-3.5) -- (1,-5.25) -- (2,-3.5) -- (3,-5.25)-- (4,-3.5) -- (5,-5.25)-- (6,-3.5) -- (7,-5.25) -- (8,-3.5) -- (9,-5.25) -- (10,-3.5) -- (11,-5.25) -- (12,-3.5) -- (13,-5.25) -- (14,-3.5);

\draw (0,-5.25) -- (14, -5.25);

\node[above right] at (1,1.75)  {\fontsize{7}{5}\selectfont$ \,\, [0,3,-2]$};
\node[above right] at (3,1.75)   {\fontsize{7}{5}\selectfont$ \,\,  [0,3,-1]$};
\node[above right] at (5,1.75)  {\fontsize{7}{5}\selectfont$ \,\,  [0,3,0]$};
\node[above right] at (7,1.75)  {\fontsize{7}{5}\selectfont$  \,\, [0,3,1]$};
\node[above right] at (9,1.75)  {\fontsize{7}{5}\selectfont$\,\,  [0,3,2]$};
\node[above right] at (11,1.75)  {\fontsize{7}{5}\selectfont$\,\,  [0,3,3]$};

\node[above right] at (0,0)  {\fontsize{7}{5}\selectfont$ \,\, [0,2,-3]$};
\node[above right] at (2,0)  {\fontsize{7}{5}\selectfont$ \,\,  [0,2,-2]$};
\node[above right] at (4,0)  {\fontsize{7}{5}\selectfont$ \,\,  [0,2,-1]$};
\node[above right] at (6,0)  {\fontsize{7}{5}\selectfont$  \,\, [0,2,0]$};
\node[above right] at (8,0)  {\fontsize{7}{5}\selectfont$\,\,  [0,2,1]$};
\node[above right] at (10,0)  {\fontsize{7}{5}\selectfont$\,\,  [0,2,2]$};
\node[above right] at (12,0)  {\fontsize{7}{5}\selectfont$\,\,  [0,2,3]$};

\node[above right] at (1,-1.75)  {\fontsize{7}{5}\selectfont$\,\,  [0,1,-3]$};
\node[above right] at (3,-1.75)  {\fontsize{7}{5}\selectfont$\,\,  [0,1,-2]$};
\node[above right] at (5,-1.75)  {\fontsize{7}{5}\selectfont$\,\,  [0,1,-1]$};
\node[above right] at (7,-1.75)  {\fontsize{7}{5}\selectfont$\,\,  [0,1,0]$};
\node[above right] at (9,-1.75)  {\fontsize{7}{5}\selectfont$\,\,  [0,1,1]$};
\node[above right] at (11,-1.75)  {\fontsize{7}{5}\selectfont$\,\,  [0,1,2]$};

\node[above right] at (0,-3.5)  {\fontsize{7}{5}\selectfont$ \,\, [0,0,-4]$};
\node[above right] at (2,-3.5)  {\fontsize{7}{5}\selectfont$ \,\,  [0,0,-3]$};
\node[above right] at (4,-3.5)  {\fontsize{7}{5}\selectfont$ \,\,  [0,0,-2]$};
\node[above right] at (6,-3.5)  {\fontsize{7}{5}\selectfont$  \,\, [0,0,-1]$};
\node[above right] at (8,-3.5)  {\fontsize{7}{5}\selectfont$\,\,  [0,0,0]$};
\node[above right] at (10,-3.5)  {\fontsize{7}{5}\selectfont$\,\,  [0,0,1]$};
\node[above right] at (12,-3.5)  {\fontsize{7}{5}\selectfont$\,\,  [0,0,2]$};

\node[above right] at (1,-5.25)  {\fontsize{7}{5}\selectfont$\,\,  [0,-1,-4]$};
\node[above right] at (3,-5.25)  {\fontsize{7}{5}\selectfont$\,\,  [0,-1,-3]$};
\node[above right] at (5,-5.25)  {\fontsize{7}{5}\selectfont$\,\,  [0,-1,-2]$};
\node[above right] at (7,-5.25)  {\fontsize{7}{5}\selectfont$\,\,  [0,-1,-1]$};
\node[above right] at (9,-5.25)  {\fontsize{7}{5}\selectfont$\,\,  [0,-1,0]$};
\node[above right] at (11,-5.25)  {\fontsize{7}{5}\selectfont$\,\,  [0,-1,1]$};

\end{tikzpicture}
\end{center}

\end{figure}

\begin{prop}
\label{equivalence reflections}
Let $\Gamma$ and $\Gamma'$ be two isomorphic tiled orders whose convex polytopes $C_\Gamma$ and $C_{\Gamma'}$ are in $\mathcal{A}$. Then $[\Gamma]=[\Gamma']$ are in the same equivalence class from Definition \ref{class} if and only if $C_\Gamma$ and $C_{\Gamma'}$ are reflection equivalent.
\end{prop}

\begin{proof}
Suppose $[\Gamma]=[\Gamma']$, then there exists $\sigma \in S_n$ such that
\begin{equation*}
\begin{array}{cllccll}
m_{ij\ell}'& = & m_{\sigma(i)\sigma(j)\sigma(\ell)}  & \text{and} &  t_i'& \equiv  &t_{\sigma(i)}  \pmod{n}
\end{array}
\end{equation*}
for all $i,j,\ell \le n$. Let  $\eta=(\delta_{\sigma^{-1}(i)j})$  and  $\Gamma''=\eta\Gamma'\eta^{-1}$. Since $t(\xi)=0$, by Lemma \ref{type0}  $\eta$ will act on the apartment by a product of reflections and therefore $C_\Gamma''$ and $C_{\Gamma'}$ are reflection equivalent. 

By Corollary \ref{movep} and the equation above, the structural invariants and types of $\Gamma''$ are 
\[ m''_{ij\ell}= m'_{\sigma^{-1}(i)\sigma^{-1}(j)\sigma^{-1}(\ell)} = m_{ij\ell} \quad \text{and} \quad t''_i \equiv t'_{\sigma^{-1}(i)} \equiv t_i \pmod{n}.\]

Therefore $[\Gamma'']=[\Gamma]$. Moreover, since $\Gamma''$ and $\Gamma'$ have equal structural invariants, according to  Corollary \ref{translation} $C_{\Gamma''}$ must be a translation of $C_{\Gamma}$ by some diagonal matrix. Since $t_i \equiv t_i'' \pmod{n}$, the type of such a diagonal matrix must be zero, and  by Lemma \ref{type0} the matrix will act on the apartment by a product of reflections. This implies that  $C_{\Gamma}$ and $C_{\Gamma''}$ are reflection equivalent, and by transitivity so are $C_\Gamma$ and $C_{\Gamma'}$.

Now we prove the converse, and assume that $C_\Gamma$ and $C_{\Gamma'}$ are reflection equivalent.  Let the product of reflections sending $C_\Gamma$ to $C_{\Gamma'}$ correspond to the monomial matrix $\xi=(\bm{\pi}^{\beta_i}\delta_{\sigma(i)j})$ with $t(\xi)\equiv 0 \pmod{n}$ such that $\Gamma'=\xi\Gamma\xi^{-1}$. By Lemma \ref{action lem},  $$[\xi \Gamma \xi^{-1}] = [(m_{\sigma(i)\sigma(j)\sigma(\ell)}), (t_{\sigma(1)}, \dots, t_{\sigma(n)})]$$ and we are done.
\end{proof}

Therefore, the equivalence classes described above partition the convex polytopes congruent to $C_{\Gamma}$ into classes of reflection equivalent convex polytopes. We continue by investigating the number of such equivalence classes.

\begin{lem}
\label{class description}
Let $\Gamma$ be a tiled order with  tuple $(m_{ij\ell})$ of structural invariants in lexicographical order, and ordered tuple of types of distinguished vertices $(t_1,t_2, \dots, t_n)$. For any $s \in \Z$, let $\xi_s \coloneqq \text{diag}(\bm{\pi}^s,1,\dots,1)$ and $\Gamma_s \coloneqq \xi_s \Gamma \xi_s^{-1}$. Then there are at most $n$   reflection classes of polytopes congruent to $C_\Gamma$,  corresponding to the classes of orders 
\begin{center}
$[\Gamma]= [\Gamma_0]=[(m_{ij\ell}), (t_1,t_2, \dots, t_n)] \qquad \qquad \qquad \qquad \qquad \qquad \qquad \,$

$\,\,[\Gamma_1]=[(m_{ij\ell}), (t_1+1, t_2+1, \dots, t_n+1)]  \qquad \qquad \qquad \quad$

$\,\, [\Gamma_2]=[(m_{ij\ell}), (t_1+2, t_2+2, \dots, t_n+2)] \qquad \qquad \qquad \quad$

$\vdots \quad $

 $[\Gamma_{n-1}]=[(m_{ij\ell}), (t_1+n-1, t_2+n-1, \dots, t_n+n-1)].\qquad$
\end{center}
\end{lem}

\begin{proof}
The fact that $[\Gamma_s]=[(m_{ij\ell}), (t_1+s, t_2+s, \dots, t_n+s)]$ follows from Lemma \ref{action lem}.
We show that any order $\Gamma'$ isomorphic to $\Gamma$, and with convex polytope  $C_{\Gamma'}$ in $\mathcal{A}$, belongs to one of the classes enumerated above.  If $\Gamma' \cong \Gamma$, the main theorem in \cite{tiledisom} gives a monomial matrix $\xi=(\pi^{\beta_i}\delta_{\tau(i)j}), \, \tau\in S_n$ such that  $\Gamma'=\xi\Gamma\xi^{-1}$. Let $\eta=(\delta_{\tau^{-1}(i)j})$. By Lemma \ref{type0} and Proposition \ref{equivalence reflections}, $[\Gamma']=[\eta \Gamma' \eta^{-1}]$. Let  $\Gamma''\coloneqq  \eta \Gamma' \eta^{-1}= (\eta \xi) \Gamma (\eta \xi)^{-1}$. Since the product $\eta \xi$ is a diagonal matrix with $t(\eta \xi) \equiv t(\xi) \pmod{n}$,  by Lemma \ref{action lem}  the equivalence class $[\Gamma'']$ is determined by the data   $$m''_{ij\ell}=m_{ij\ell} \quad \text{and} \quad t_i'' \equiv t_i+t(\xi) \pmod{n} \quad \text{for all} \quad i,j,\ell \le n.$$
Therefore, $[\Gamma']=[\Gamma'']$ corresponds to the  reflection class given by $[\Gamma_{t(\xi)}]$. 
\end{proof}

\begin{example} Let $\Gamma$ be the tiled order with exponent matrix $M_\Gamma=\left( \begin{array}{ccc}
0 & 1 &1 \\
0 & 0 & 1\\
0 & 1& 0\end{array} \right)$, its convex polytope denoted in Figure \ref{reflection classes} by the blue diamond. The types of its distinguished vertices are $(0,2,2)$. We see other two reflection classes given by $\Gamma_1$ with $M_{\Gamma_1}=\left( \begin{array}{ccc}
0 & 0 &1 \\
1 & 0 & 2\\
0 & 0& 0\end{array} \right)$, types $(1,0,0)$ and convex polytope in yellow, and $\Gamma_2$ with $M_{\Gamma_2}=\left( \begin{array}{ccc}
0 & 0 &0 \\
1 & 0 & 1\\
1 & 1& 0\end{array} \right)$,  types $(2,1,1)$ and convex polytope in green. Note that the three polytopes are in distinct reflection classes.

\end{example}

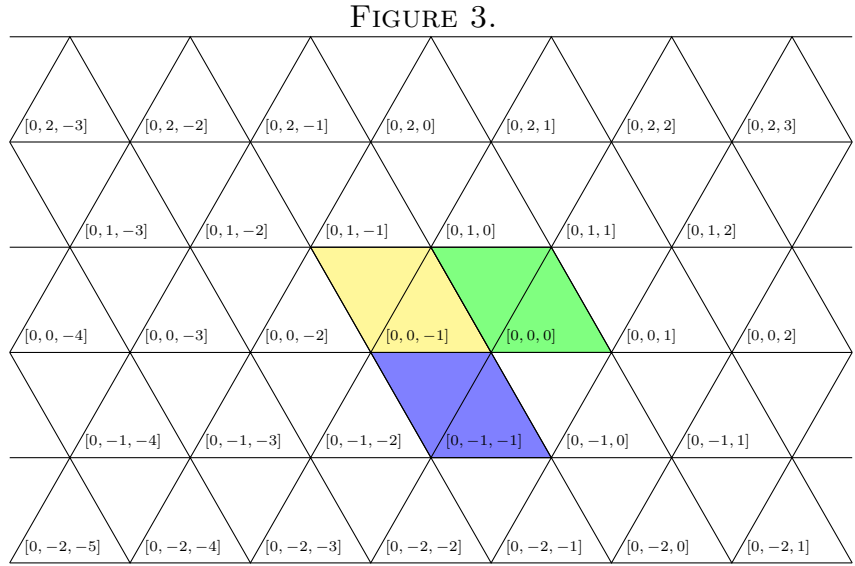
\begin{figure}[h]
\caption{}
\label{reflection classes}

\smallskip

\begin{tikzpicture}[scale=0.8, every node/.style={transform shape}]

\draw[fill=blue!50] (8,-3.5)--(9,-5.25)--(7,-5.25) -- (6,-3.5) -- (8,-3.5);
\draw[fill=yellow!50] (8,-3.5)--(6,-3.5)--(5,-1.75) -- (7,-1.75) -- (8,-3.5);
\draw[fill=green!50] (8,-3.5)--(7,-1.75)--(9,-1.75) -- (10,-3.5) -- (8,-3.5);

\draw (0,0) -- +(1,1.75) -- (2,0) -- +(1,1.75)-- (4,0) -- +(1,1.75)-- (6,0) -- +(1,1.75) -- (8,0) -- +(1,1.75) -- (10,0) -- +(1,1.75) -- (12,0) -- +(1,1.75) -- (14,0);

\draw (0,1.75) -- (14, 1.75);

\draw (0,0) -- (14, 0);

\draw (0,0) -- (1,-1.75) -- (2,0) -- (3,-1.75)-- (4,0) -- (5,-1.75)-- (6,0) -- (7,-1.75) -- (8,0) -- (9,-1.75) -- (10,0) -- (11,-1.75) -- (12,0) -- (13,-1.75) -- (14,0);

\draw (0,-1.75) -- (14, -1.75);

\draw (0,-3.5) -- +(1,1.75) -- (2,-3.5) -- +(1,1.75)-- (4,-3.5) -- +(1,1.75)-- (6,-3.5) -- +(1,1.75) -- (8,-3.5) -- +(1,1.75) -- (10,-3.5) -- +(1,1.75) -- (12,-3.5) -- +(1,1.75) -- (14,-3.5);

\draw (0,-3.5) -- (14, -3.5);

\draw (0,-3.5) -- (1,-5.25) -- (2,-3.5) -- (3,-5.25)-- (4,-3.5) -- (5,-5.25)-- (6,-3.5) -- (7,-5.25) -- (8,-3.5) -- (9,-5.25) -- (10,-3.5) -- (11,-5.25) -- (12,-3.5) -- (13,-5.25) -- (14,-3.5);

\draw (0,-5.25) -- (14, -5.25);

\draw (0,-7) -- +(1,1.75) -- (2,-7) -- +(1,1.75)-- (4,-7) -- +(1,1.75)-- (6,-7) -- +(1,1.75) -- (8,-7) -- +(1,1.75) -- (10,-7) -- +(1,1.75) -- (12,-7) -- +(1,1.75) -- (14,-7);

\draw (0,-7) -- (14, -7);

\node[above right] at (0,0)  {\fontsize{7}{5}\selectfont$ \,\, [0,2,-3]$};
\node[above right] at (2,0)  {\fontsize{7}{5}\selectfont$ \,\,  [0,2,-2]$};
\node[above right] at (4,0)  {\fontsize{7}{5}\selectfont$ \,\,  [0,2,-1]$};
\node[above right] at (6,0)  {\fontsize{7}{5}\selectfont$  \,\, [0,2,0]$};
\node[above right] at (8,0)  {\fontsize{7}{5}\selectfont$\,\,  [0,2,1]$};
\node[above right] at (10,0)  {\fontsize{7}{5}\selectfont$\,\,  [0,2,2]$};
\node[above right] at (12,0)  {\fontsize{7}{5}\selectfont$\,\,  [0,2,3]$};

\node[above right] at (1,-1.75)  {\fontsize{7}{5}\selectfont$\,\,  [0,1,-3]$};
\node[above right] at (3,-1.75)  {\fontsize{7}{5}\selectfont$\,\,  [0,1,-2]$};
\node[above right] at (5,-1.75)  {\fontsize{7}{5}\selectfont$\,\,  [0,1,-1]$};
\node[above right] at (7,-1.75)  {\fontsize{7}{5}\selectfont$\,\,  [0,1,0]$};
\node[above right] at (9,-1.75)  {\fontsize{7}{5}\selectfont$\,\,  [0,1,1]$};
\node[above right] at (11,-1.75)  {\fontsize{7}{5}\selectfont$\,\,  [0,1,2]$};

\node[above right] at (0,-3.5)  {\fontsize{7}{5}\selectfont$ \,\, [0,0,-4]$};
\node[above right] at (2,-3.5)  {\fontsize{7}{5}\selectfont$ \,\,  [0,0,-3]$};
\node[above right] at (4,-3.5)  {\fontsize{7}{5}\selectfont$ \,\,  [0,0,-2]$};
\node[above right] at (6,-3.5)  {\fontsize{7}{5}\selectfont$  \,\, [0,0,-1]$};
\node[above right] at (8,-3.5)  {\fontsize{7}{5}\selectfont$\,\,  [0,0,0]$};
\node[above right] at (10,-3.5)  {\fontsize{7}{5}\selectfont$\,\,  [0,0,1]$};
\node[above right] at (12,-3.5)  {\fontsize{7}{5}\selectfont$\,\,  [0,0,2]$};

\node[above right] at (1,-5.25)  {\fontsize{7}{5}\selectfont$\,\,  [0,-1,-4]$};
\node[above right] at (3,-5.25)  {\fontsize{7}{5}\selectfont$\,\,  [0,-1,-3]$};
\node[above right] at (5,-5.25)  {\fontsize{7}{5}\selectfont$\,\,  [0,-1,-2]$};
\node[above right] at (7,-5.25)  {\fontsize{7}{5}\selectfont$\,\,  [0,-1,-1]$};
\node[above right] at (9,-5.25)  {\fontsize{7}{5}\selectfont$\,\,  [0,-1,0]$};
\node[above right] at (11,-5.25)  {\fontsize{7}{5}\selectfont$\,\,  [0,-1,1]$};

\node[above right] at (0,-7)  {\fontsize{7}{5}\selectfont$ \,\, [0,-2,-5]$};
\node[above right] at (2,-7)  {\fontsize{7}{5}\selectfont$ \,\,  [0,-2,-4]$};
\node[above right] at (4,-7)  {\fontsize{7}{5}\selectfont$ \,\,  [0,-2,-3]$};
\node[above right] at (6,-7)  {\fontsize{7}{5}\selectfont$  \,\, [0,-2,-2]$};
\node[above right] at (8,-7)  {\fontsize{7}{5}\selectfont$\,\,  [0,-2,-1]$};
\node[above right] at (10,-7)  {\fontsize{7}{5}\selectfont$\,\,  [0,-2,0]$};
\node[above right] at (12,-7)  {\fontsize{7}{5}\selectfont$\,\,  [0,-2,1]$};

\end{tikzpicture}

\end{figure}

Therefore, there are at most $n$ equivalence classes of tiled orders isomorphic to $\Gamma$. However, not all of the equivalence classes in Lemma \ref{class description} are always distinct.

\begin{lem}
\label{norm class}

\begin{enumerate}
\item Let $\xi \in \mathcal{N}(\Gamma)$ be a monomial matrix. Then $[(m_{ij\ell}), (t_1, \dots, t_n)]=[(m_{ij\ell}), (t_1+\ell \cdot t(\xi), \dots, t_n+\ell \cdot t(\xi))]$ for any $\ell \in \Z$.

\item  If $[\Gamma_i]=[\Gamma_{i+r}]$, then $[\Gamma_i]=[\Gamma_{i+\ell \cdot r}]$ for all $\ell \in \Z$. 
\end{enumerate}
\end{lem}

\begin{proof}

(1) Write $\xi=\xi_d\cdot p_\sigma$ as a product of a diagonal matrix $\xi_d$ and a permutation matrix $p_\sigma$. Then  $t(p_\sigma)\equiv 0 \pmod{n}$ and $t(\xi_d)\equiv t(\xi) \pmod{n}$. Since $\Gamma=\xi \Gamma \xi^{-1}$ and $p_\sigma$ acts on $\mathcal{A}$ as a product of reflections, by Lemma \ref{action lem} and Lemma \ref{type0} we have 
\[ [\Gamma]= [\xi \Gamma \xi^{-1}]=\xi_d [p_\sigma \Gamma p_{\sigma}^{-1}]=\xi_d[\Gamma]=[\xi_d \Gamma \xi_d^{-1}], \]
which in terms of invariants gives $$[(m_{ij\ell}), (t_1, \dots, t_n)]=[(m_{ij\ell}), (t_1+t(\xi), \dots, t_n+t(\xi))].$$

The claim follows since $\xi^\ell \in \mathcal{N}(\Gamma)$ for any $\ell \in \Z$ and we can repeat the process.

(2) By Proposition \ref{equivalence reflections},  $[\Gamma_i]=[\Gamma_{i+r}]$ if and only if $C_{\Gamma_i}$ and $C_{\Gamma_{i+r}}$ are reflection equivalent, so by Lemma \ref{type0} there exists  a monomial matrix $\eta$ of type $0$ such that $\eta \Gamma_i \eta^{-1}=\Gamma_{i+r}$. With the notation as in Lemma \ref{class description}, note that $ \xi_{i+r} \xi_i^{-1} \eta^{-1} \in \mathcal{N}(\Gamma_{i+r})$, so by (a) we have $[\Gamma_i]=[\Gamma_{i+r}]=[\Gamma_{i+ \ell \cdot r}]$ for all $ \ell \in \Z$. 
\end{proof}

\begin{thm}
\label{n classes}
Let $\Gamma$ be a tiled order, and $\Gamma_i$ and their corresponding classes as defined in Lemma \ref{class description}. Then the following are equivalent:

\begin{enumerate}[(a)]

\item There are $d$ distinct equivalence classes.

\item $d$ is the smallest among $\{1, 2, \dots, n\}$ such that  $[\Gamma_s]=[\Gamma_t]$ whenever $s \equiv t \pmod{d}$.

\item $d$  is the smallest among $\{1, 2, \dots, n\}$ such that  $[\Gamma_0]=[\Gamma_d]$.

\item $\nr(\mathcal{N}(\Gamma))=(k^\times)^d R^\times$.
\end{enumerate}

\end{thm}

\begin{proof}

$(a) \implies (b)$ Suppose there are $d$ distinct equivalence classes. Note that always   $[\Gamma_{i+\ell \cdot n}]=[(m_{ij\ell}), (t_1+i+\ell \cdot n, t_2+i+\ell \cdot n,\dots, t_n+i+\ell \cdot n)]=[(m_{ij\ell}), (t_1+i,  \dots, t_n+i)]=[\Gamma_i]$. Then the result follows immediately if $d=1$ (then $[\Gamma_s]=[\Gamma_t]$ for all $s, t \in \Z$), or $d=n$ (then $[\Gamma_s]$ are all distinct for $0 \le s \le n-1$). 

Suppose $1<d<n$. Then there exist $ 0 \le i < j \le n-1$ such that $[\Gamma_i]=[\Gamma_j]$, and we may take $i,j$ such that $r = |j-i|$ is minimal. Lemma  \ref{norm class}  then gives $[\Gamma_i]=[\Gamma_{i+ \ell \cdot r}]$ for all $ \ell \in \Z$.  On the other hand, since $[\Gamma_i]=[\Gamma_j]$, then clearly $[\Gamma_{i+t}]=[\Gamma_{j+t}]$ for all $t \in \Z$ so again by Lemma \ref{norm class} we get  $[\Gamma_s]=[\Gamma_t]$ if $s \equiv t \pmod{r}$. Then our claim holds since $1 \le r \le n$ was chosen minimal.

$(b) \implies (c)$  Immediate. 

$(c) \implies (a)$  We clearly have $[\Gamma_i]=[\Gamma_{i+d}]$ for all $i \in \Z$, and Lemma \ref{norm class}  gives   $[\Gamma_0]=[\Gamma_{\ell \cdot d}]$ for all $ \ell \in \Z$. Therefore, there are  $d$ distinct equivalence classes.

$(c) \implies (d)$ By Proposition \ref{equivalence reflections}, we have a monomial matrix $\eta$ with $t(\eta)\equiv 0 \pmod{n}$ and $\Gamma=\eta \Gamma_d \eta^{-1}=\eta \xi_d \Gamma \xi_d^{-1} \eta^{-1}$, so $\eta \xi_d \in \mathcal{N}(\Gamma)$ has $t(\eta \xi_d) \equiv d \pmod{n}$. By Proposition \ref{typesnrd}, $(k^\times)^d R^\times \subseteq \nr(\mathcal{N}(\Gamma))$. On the other hand, if there exists a monomial matrix  $\xi_\sigma \in \mathcal{N}(\Gamma)$ with $t(\xi_\sigma) \equiv r$ with $0 <r<d$,  Lemma \ref{norm class} gives $[\Gamma_0]=[\Gamma_r]$ which contradicts that $d$ is minimal. Therefore, $(k^\times)^d R^\times = \nr(\mathcal{N}(\Gamma))$.

$(d) \implies (c)$ By  Proposition \ref{typesnrd}, $1 \le d \le n$ is minimal among the types $t(\xi_\sigma)$ for $\xi_\sigma \in \mathcal{N}(\Gamma)$, so Lemma \ref{norm class} gives $[\Gamma_0]=[\Gamma_d]$ with $d$ minimal.
\end{proof}

We would like to use our results  to compute $\nr(\mathcal{N}(\Gamma))$ for any given tiled order.  We have the following algorithm:

%
%
%

\begin{algorithm}
\caption{Algorithm for determining the number of reflection classes for $\Gamma \subset M_n(D)$}\label{alg1}
\begin{algorithmic}[1]
\Procedure{NumberOfReflectionClasses}{$\Gamma$}							
\State Compute the structural invariants  $m_{ij\ell}$ and types of distinguished vertices $t_i$.
\State Compute a subgroup $ G \subseteq S_n$ containing $H$ as in Proposition \ref{thmants} using steps (1)-(4) in the Algorithm in \cite{ants}  \Comment{This step is optional, but polynomial in time and it can reduce the time needed. Otherwise, we can take $G=S_n$.}
\State Compute the divisors $d_i|n$  in increasing order. 
\State Let $d \coloneqq d_1=1$.
\Repeat                                                                                        \Comment{For each divisor $d_i$, we will check whether $[\Gamma_0]=[\Gamma_{d_i}]$.}
 \State Find all the permutations $\sigma \in G$ that decompose into products of disjoint cycles with length not coprime to $r$, such that  $t_j+d_i \equiv t_{\sigma(j)} \pmod{n}$ for all $1 \le j \le n$.
\State  For each $\sigma$ found in the previous step, check whether $m_{ij\ell}=m_{\sigma(i)\sigma(j)\sigma(\ell)}$ for all $i,j,\ell \le n$. 
\If{there exists at least one such $\sigma$}  \textbf{break.}  \Comment{We either found the divisor $d$ equal to the number of reflection classes,}
\Else  \quad  $d=d_{i+1}$  
\EndIf
\Until{$d=n$.} \Comment{or we exhausted all divisors.}
\State \textbf{return} $d$
\EndProcedure
\end{algorithmic}
\end{algorithm}

We illustrate Algorithm \ref{alg1} with a couple of examples.

\begin{example}
\label{p1}
Consider the tiled order $\Gamma$  with exponent matrix $ M_{\Gamma} = \begin{pmatrix} 0&1&1&2\\
2&0&2&2\\
2&1&0&1\\
1&1&0&0\end{pmatrix}$. 

The types are given by the tuple $(1,3,3,1)$, and we omit to write down the structural invariants due to space constraints. We skip the optional step, and let $G=S_4$. The proper divisors of $4$ are $d_1=1$ and $d_2=2$.   The types are given by the tuple $(1,3,3,1)$.

Let $d=1$.  We want $\sigma \in G$ for which  $t_j+1 \equiv t_{\sigma(j)} \pmod{4}$ for all $j \le 4$. However, since none of the vertices has type $2$ and $t_1+1=2$, there is no such $\sigma$.

Let $d=2$.  We check for permutations in $G$ decomposing into disjoint cycles of length not coprime to $4$ for which  $t_j+2 \equiv t_{\sigma(j)} \pmod{4}$ for all $j \le 4$. The eligible permutations are $(1243), (12)(34), (1342)$ and $(13)(42)$.  Note that $m_{123}=2 \neq m_{241}=1$, so $(1243)$ does not apply. However, we can check that $m_{ij\ell}=m_{\sigma(i)\sigma(j)\sigma(\ell)}$ holds for all $i,j,\ell \le 4$ when $\sigma=(12)(34)$. Therefore, $d=2$, and $\nr(\mathcal{N}(\Gamma))=(k^\times)^2R^\times$.

Note that if we computed $G$ in step $3$, we would get $G=\{(), (12), (34),  (12)(34)\}$, and we would only have to consider the permutation $(12)(34)$.
\end{example}

\begin{example}  As discussed in \cite[page 76]{brown}, any two chambers can be connected by reflections. For example, the polytopes in Figure \ref{general} are chambers. One example of an order whose polytope is a chamber is $\Gamma$ with upper triangular exponent matrix $M_\Gamma= (a_{ij})$ where $a_{ij}=1$ if $i < j$, and $a_{ij}=0$ otherwise. One can  check using Algorithm \ref{alg1} that there is only one reflection class.
\end{example}

\begin{proof}[Proof of correctness of Algorithm \ref{alg1}]
By Theorem \ref{n classes}, there are $d$ equivalence classes if and only if $[\Gamma_0]=[\Gamma_d]$ with $1 \le d \le n$ minimal. Note that if $\Gamma=\Gamma_0$ has structural invariants $m_{ij\ell}$ and types $t_i$, by Corollary \ref{movep}, $\Gamma_d$ has structural invariants $m_{ij\ell}$ and types $t_i+d_i$. But then $[\Gamma_0]=[\Gamma_d]$ if and only if there exists $\sigma \in S_n$ for which $m_{ij\ell}=m_{\sigma(i)\sigma(j)\sigma(\ell)}$ and $t_i+d\equiv t_{\sigma(\ell)} \mod{n}$, conditions which precisely correspond to the repeated step.
\end{proof}

As we could see in Example \ref{p1}, even for small $n$ the task of finding the number of reflection classes can be quite involved, and the more information we can get about $\mathcal{N}(\Gamma)$, such as the group $G$, the better. Fortunately, the algorithm above reduces to a very simple case when $n=p$ is prime.

\begin{algorithm}
\caption{Algorithm for determining the number of reflection classes for $\Gamma \subset M_p(D)$}\label{alg2}
\begin{algorithmic}[1]
\Procedure{NumberOfReflectionClassesPrime}{$\Gamma$}							
\State Compute the structural invariants  $m_{ij\ell}$ and types of distinguished vertices $t_i$. 
\If{all the types $t_i$ are distinct}
\State Find the unique $p$-cycle $\sigma \in S_p$ such that $t_j+1 \equiv t_{\sigma(j)} \pmod{p}$ for all $1 \le j \le p$.
\If{ $m_{ij\ell}=m_{\sigma(i)\sigma(j)\sigma(\ell)}$ for all $i,j,\ell \le p$} {\textbf{return} $1$}
\Else \quad {\textbf{return} $p$}
\EndIf
\Else \quad {\textbf{return} $p$}
\EndIf
\EndProcedure
\end{algorithmic}
\end{algorithm}

\begin{proof}[Proof of correctness of Algorithm \ref{alg2}]
First, we need to confirm that if the types $t_i$ are not all distinct, then there are $p$ distinct reflection classes. By Theorem \ref{n classes} and Proposition \ref{typesnrd}, the number of equivalence classes divides $p$, so it's either $1$ or $p$. Suppose there was only one equivalence class, then by Proposition \ref{typesnrd} there exists $\xi_\sigma \in \mathcal{N}(\Gamma)$ a monomial matrix with $t(\xi)\equiv 1 \pmod{p}$. By Corollay \ref{movep}, $t_i\equiv t(\xi)+t_{\sigma(i)}$ for all $i$, which means all vertices have distinct types, which contradicts our assumption.

Next, suppose all types are distinct. By Lemma \ref{disjoint cycles}, we only need to consider $p$-cycles. Then there is clearly a unique $\sigma \in S_p$ such that $t_j+1 \equiv t_{\sigma(j)} \pmod{p}$ for all $1 \le j \le p$.  If $m_{ij\ell}=m_{\sigma(i)\sigma(j)\sigma(\ell)}$ for all $i,j,\ell \le p$, then $[\Gamma_0]=[\Gamma_1]$ and by Theorem \ref{n classes} there is only one reflection class. Otherwise, there must be $p$ such classes.
\end{proof}

\section{Type numbers}
\label{typs}

 Recall our notation. Let  $K$ be  a number field with ring of integers $\mathcal{O}_K$ and  set of places $\text{Pl}(K)$. Let  $A$ be a central simple algebra over $K$  such that either the degree of $A$ is $n \ge 3$, or $n=2$ and $A$ is not a totally definite quaternion algebra, so strong approximation holds in $A$. Denote by  $\Omega \subset \text{Pl}(K)$ the finite set of real places of $K$ ramifying in $A$. Consider $\Gamma$ an $\mathcal{O}_K$-order in $A$, such that $\Gamma_\nu$ is tiled at each finite place $\nu \in \text{Pl}(K)$. Note that at all but finitely many primes, $\Gamma_\nu$ is maximal. We denote by $K_\nu$ and $\mathcal{O}_\nu$ the completions of $K$, and respectively $\mathcal{O}_K$,  at a place $\nu \in \text{Pl}(K)$. When $\nu$ is finite, $K_\nu$ is an extension of a $p$-adic field, when $\nu$ is infinite and real $K_\nu=\R$, and when $\nu$ is infinite and complex $K_\nu=\C$.  Let $A_\nu:=K_\nu \otimes_K A$ and $\Gamma_\nu := \mathcal{O}_\nu \otimes_R \Gamma$. By Artin-Wedderburn, $A_\nu \cong M_{n_\nu}(D_\nu)$, where $D_\nu$ is a central division algebra of degree $n/n_\nu$ over $K_\nu$.    If $\nu$ is an infinite place, we set $\mathcal{O}_\nu:=K_\nu$ and  $\Gamma_\nu:=A_\nu$.

The type number $G(\Gamma)$ of $\Gamma$ is the number of isomorphisms classes of orders locally isomorphic to $\Gamma$, or equivalently, the number of double cosets $A^\times\backslash J_A/\prod'_{\nu} \mathcal{N}(\Gamma_\nu)$. We will refer to this set of cosets   as the \textit{genus} of $\Gamma$. Since strong approximation holds in $A$,   we get the bijection from Equation (\ref{types})
\[
  A^\times\backslash J_A/\sideset{}{'}\prod_{\nu} \mathcal{N}(\Gamma_\nu) \leftrightarrow  J_K/K^\times \nr(\sideset{}{'}\prod_{\nu} \mathcal{N}(\Gamma_\nu)).
\]


We recall the options for the idelic normalizer of the completion $\Gamma_\nu$. Denote by $S\coloneqq S_\infty - \Omega$. Then $A \cong M_{n/2}(\mathbb{H})$ for $\nu \in \Omega$ and $A \cong M_n(K_\nu)$ for $\nu \in S$, which will determine $\nr(\mathcal{N}(\Gamma_\nu))$ at all places $\nu \in S_\infty$. Now suppose $\nu$ is finite. Then for all but finitely many places we have $A_\nu \cong M_n(K_\nu)$ and $\Gamma_\nu\cong M_n(\mathcal{O}_\nu)$, in which case  by Corollary  \ref{max order norm},  $\nr(\mathcal{N}(\Gamma_\nu))=(K_\nu^\times)^n \mathcal{O}_\nu^\times$. At the finitely many remaining cases, we have $A_\nu \cong M_{n_\nu}(D_\nu)$ and $\Gamma_\nu$ a tiled order in $A_\nu$, and Proposition \ref{typesnrd} gives $\nr(\mathcal{N}(\Gamma_\nu)) =(K_\nu^\times)^{d_\nu} \mathcal{O}_\nu^\times$, where $d_\nu|n_\nu$ and $n_\nu|n$. Note that it is possible that $d_\nu=n_\nu=n$, but we can put such cases together with the case above, and let $T$ be  the remaining set of finite places such that $\nr(\mathcal{N}(\Gamma_\nu))=(K_\nu^\times)^{d_\nu}\mathcal{O}_\nu^\times$ where $d_\nu \ne n$. To summarize, we have
\[ \nr(\mathcal{N}(\Gamma_\nu)) = \begin{cases} 
      \R_+^\times &   \nu \in \Omega\\
      K_\nu^\times &  \nu \in S \\
(K_\nu^\times)^{d_\nu} \mathcal{O}_\nu^\times & \nu \in T \text{ with } d_\nu \ne n, d_\nu | n,\\
(K_\nu^\times)^n\mathcal{O}_\nu^\times &  \nu \not\in S_\infty \cup T.
      \\
   \end{cases}
\]

Therefore, we want the size of the  idelic quotient 
\begin{equation}
\label{general tiled}
J_K/K^\times \prod_{\nu \in \Omega} \R_+^\times \prod_{\nu \in S} K_\nu^\times \prod_{\nu \in T} (K_\nu^\times)^{d_\nu}\mathcal{O}_\nu^\times \sideset{}{'}\prod_{\nu \not \in S_\infty \cup T} (K_\nu^\times)^n\mathcal{O}_\nu^\times.
\end{equation}

We first bound $G(\Gamma)$ above. All maximal orders in $A$ are locally isomorphic, so the type number of all maximal orders are equal; we denote this number by $G_{max}$. 

\begin{prop}
\label{divides tower}
Let $A$ be a central simple algebra of degree $n \ge 2$  over a number field $K$ such that $A$ is not a totally definite quaternion algebra. Let $\mathcal{O}_K$ be the  ring of integers of $K$,  $\Omega$ the set of real places ramifying in $A$, and $\Cl_\Omega(K)$ the ray class group for $\Omega$. Let $G_{max}$ be the type number of maximal orders in $A$. Given an everywhere locally tiled order $\Gamma$ in $A$, we have $$G(\Gamma) \le G_{max} \le \, \#\Cl_\Omega(K)/\Cl_\Omega(K)^n,$$ where in particular $G(\Gamma)| G_{max}$ and $G_{max} | \, \#\Cl_\Omega(K)/\Cl_\Omega(K)^n$.
\end{prop}

\begin{proof}
Recall that $\Cl_\Omega(K) \cong J_K/K^\times J_{K, S, \Omega}$, where $S=S_\infty-\Omega$. Then $$\Cl_\Omega(K)^n \cong  J_K^n/(K^\times  J_{K,S,\Omega}\cap J_K^n) \cong J_K^nK^\times J_{K,S,\Omega}/K^\times J_{K,S,\Omega},$$ and therefore

\[ \displaystyle \begin{array}{*3{>{\displaystyle}ccl}}
\Cl_\Omega(K)/\Cl_\Omega(K)^n & \displaystyle \cong & J_K/  J_K^nK^\times J_{K,S,\Omega}  \\
						& \displaystyle \cong &  J_K/ \prod_{\nu \in \Omega} \R_+^\times \prod_{\nu \in S} K_\nu^\times  \prod_{\nu \text{ finite}} (K_\nu^\times)^n (\mathcal{O}_\nu^\times)^n.
\end{array}\]


The genus of a maximal order $\Lambda$ in $A$ will correspond to the quotient
 $$J_K/K^\times \prod_{\nu \in \Omega} \R_+^\times \prod_{\nu \in S} K_\nu^\times \prod_{\nu \in T} (K_\nu^\times)^{n_\nu}\mathcal{O}_\nu^\times \sideset{}{'}\prod_{\nu \not \in S_\infty \cup T} (K_\nu^\times)^n\mathcal{O}_\nu^\times=J_K/ \prod_{\nu \in T}(\dots, 1,  (K_\nu^\times)^{n_\nu}, 1, \dots) K^\times J_K^nJ_{K,S,\Omega}$$ where $T$ is the set of finite primes where $A_\nu \cong M_{n_\nu}(D_\nu)$ for which $n_\nu \ne n$.  

At the same time, the genus of an arbitrary order that is everywhere locally tiled is given by  $$J_K/K^\times \prod_{\nu \in \Omega} \R_+^\times \prod_{\nu \in S} K_\nu^\times \prod_{\nu \in T} (K_\nu^\times)^{d_\nu}\mathcal{O}_\nu^\times \sideset{}{'}\prod_{\nu \not \in S_\infty \cup T} (K_\nu^\times)^n\mathcal{O}_\nu^\times= J_K/\prod_{\nu \in T}(\dots, 1, (K_\nu^\times)^{d_\nu},1,\dots) K^\times J_K^n J_{K, S, \Omega}.$$


Then our claim follows from the subgroup inclusions  $$J_K^n  \le   \prod_{\nu \in T}(\dots,1,(K_\nu^\times)^{n_\nu},1,\dots) J_K^n  \le \prod_{\nu \in T} (\dots, 1, (K_\nu^\times)^{d_\nu}, 1, \dots) J_K^n. \qedhere$$
\end{proof}

%


By Proposition \ref{divides tower}, the genus of $\Gamma$ corresponds to a subgroup of $\Cl_\Omega(K)/\Cl_\Omega(K)^n$, which we would like to identify. Consider an everywhere  locally tiled order $\Gamma \subset A$. For each place $\nu \in T=\{\nu \text{ finite } : \, \nr(\mathcal{N}(\Gamma_\nu))= (K_\nu^\times)^{d_\nu} \mathcal{O}_\nu^\times, d_\nu \ne n\}$ we have an associated ideal class $[\p_\nu]$ in $\Cl_\Omega(K)$. By the Chebotarev density theorem, each ideal class in $\Cl_\Omega(K)$ contains infinitely many prime ideals, so for each prime $\p_\nu$ with $\nu \in T$ we can pick a prime $\q_\nu$  such that $[\p_\nu^{d_\nu}]=[\q_\nu]$. Let $\hat{T}=\{\q_\nu:  \nu \in T\} \cup S$. Note that $\hat{T}$ is a finite set.

\begin{thm}
\label{type tiled}
 Let $A$ be a central simple algebra of degree $n \ge 2$  over a number field $K$, such that either $ n \ge 3$, or $A$ is not a totally definite quaternion algebra. Let $\Omega$ be the set of real ramified primes in $A$, and $S=S_\infty -\Omega$. Let $\Gamma$ be an everywhere  locally tiled order in $A$, with $T$ and $\hat{T}$ the sets of places and primes  defined above. Then 
\[
G(\Gamma)=\#\Cl_{\hat{T}, \Omega}(K)/\Cl_{\hat{T}, \Omega}(K)^n .
\]
\end{thm}

\begin{proof}
%

Let $$H= \prod_{\nu \in \Omega} \R_+^\times \prod_{\nu \in S} K_\nu^\times  \prod_{\nu \in T} (K_\nu^\times)^{d_\nu} \mathcal{O}_\nu^\times  \prod_{\nu \not\in S_\infty \cup T} \mathcal{O}_\nu^\times= \prod_{\nu \in T}(\dots,1, (K_\nu^\times)^{d_\nu},1,\dots) J_{K, S, \Omega}$$ and $G=J_K/K^\times H$. Then 
$G^n \cong  J_K^n/(K^\times  H \cap J_K^n) \cong J_K^nK^\times H/K^\times H,$ and therefore $$G/G^n \cong  J_K/  J_K^nK^\times H=J_K/\prod_{\nu \in T} (\dots,1, (K_\nu^\times)^{d_\nu},1,\dots) K^\times J_K^n J_{K, S, \Omega}.$$ We identify $G$ with a subgroup of $\Cl_\Omega(K)$ as follows.  

We have a surjective homomorphism 
\[
\Cl_\Omega(K) \cong J_K/K^\times  J_{K,S, \Omega} \rightarrow  J_K/K^\times  \prod_{\nu \in T}(\dots,1, (K_\nu^\times)^{d_\nu},1,\dots) J_{K, S, \Omega},\]
and since each $K_\nu^\times$ is generated by the uniformizer $\pi_\nu$,  we can represent the kernel of the homomorphism by  $\displaystyle \frac{\langle (\dots, 1, \pi_\nu^{d_\nu}, 1, \dots ): \nu \in T \rangle  K^\times J_{K, S, \Omega}}{K^\times  J_{K, S, \Omega}}$. Each coset $ (\dots, 1, \pi_\nu^{d_\nu}, 1, \dots ) K^\times J_{K,S,\Omega}$ corresponds to the ideal class $[\p_\nu^{d_\nu}]$, so $G \cong \Cl_\Omega(K)/\langle [\p_\nu^{d_\nu}]: \nu \in T\rangle$.  But then $[\p_\nu^{d_\nu}]=[\q_\nu]$ for each $\p_\nu$, so $\Cl_\Omega(K)/\langle [\p_\nu^{d_\nu}] : \nu \in   T\rangle =\Cl_\Omega(K)/\langle [\q_\nu]: \nu \in T\rangle = \Cl_{\hat{T}, \Omega}(K)$, and therefore $G(\Lambda)=\#  \Cl_{\hat{T}, \Omega}(K)/ \Cl_{\hat{T}, \Omega}(K)^n$.
\end{proof}

The theorem is particularly appealing when the degree of the algebra is a prime number $p \ge 3$, since then the algebra does not ramify at any infinite place and we can take $\hat{T}=T$.

\begin{cor}
Let $A$ be a central simple algebra of prime degree $p \ge 3$ over a number field $K$. Let $\Gamma$ be an everywhere  locally tiled order in $A$, with $T=\{ \nu \text{finite}: \nr(\mathcal{N}(\Gamma_\nu))=K_\nu^\times \mathcal{O}_\nu^\times\}$. Then
\[
G(\Gamma)=\#\Cl_{T}(K)/\Cl_{T}(K)^p .
\]
\end{cor}

We conclude with an example where use Algorithm \ref{alg1} and Theorem \ref{type tiled} to compute the type number of a global order.
\begin{example}
We illustrate Theorem \ref{type tiled} in the case $n=4$. Let $K=\Q(a)$ where $a$ is a root of $f(x)=x^4 - 30x^2 - 1$. Then  $\Cl(K)\cong \Z/2\Z \times \Z/8\Z$ as found using the LMFDB \cite{lmfdb}. Consider the order 
$\Gamma=\begin{pmatrix} 
\mathcal{O}_K & \p_1 & \p_1 \p_2 & \p_1^2\p_2 \\
\p_1^2 & \mathcal{O}_K & \p_1^2\p_2 & \p_1^2\p_2\\
\p_1^2& \p_1 & \mathcal{O}_K & \p_1\\
\p_1 & \p_1 & \mathcal{O}_K &\mathcal{O}_K\end{pmatrix} \subseteq M_4(K),$
 where $\p_1=(5,a+2)$ and $\p_2=(7,a-2)$. Note that since $A=M_4(K)$, none of the infinite places of $K$ ramify in $A$ so $\Omega=\emptyset$ and $\Cl_\Omega(K)=\Cl(K)$. Note also that $\Gamma_\p = M_4(\mathcal{O}_\p)$ when $\p \ne \p_1, \p_2$, and both $\Gamma_{\p_1}$ and $\Gamma_{\p_2}$ are tiled. Then $\Gamma_{\p_1}$ and $\Gamma_{\p_2}$ have exponent matrices  $$ \begin{pmatrix} 0&1&1&2\\
2&0&2&2\\
2&1&0&1\\
1&1&0&0\end{pmatrix} \qquad \text{and} \qquad \begin{pmatrix} 0&0&1&1\\
0&0&1&1\\
2&1&0&1\\
1&1&0&0\end{pmatrix}.$$ In Example \ref{p1}, we have found that $\Gamma_{\p_1}$ has two reflection classes. We can use the same algorithm to see that $\Gamma_{\p_2}$ also has  2 reflection classes. Therefore, we need two  primes $\q_1$ and $\q_2$ such that $[\p_1^2]=[\q_1]$ and $[\p_2^2]=[\q_2]$. We perform the rest of the calculations  using Sage \cite{sagemath}. First, we find such primes $\q_1=(239, a+36)$ and $\q_2=(7, a^3-33a)$. Letting $T=\{\q_1, \q_2\} \cup S_\infty$,  we get $\Cl_T(K) \cong \Z/2\Z \times \Z/2\Z$, so $\Cl_T(K)/\Cl_T(K)^4 \cong \Z/2\Z \times \Z/2\Z$. Therefore, the type number $G(\Gamma)=4$.
\end{example}

\section{Acknowledgements}
This paper incorporates material from the author's Ph.D. thesis \cite{thesis}. Many thanks to my  advisor Thomas R. Shemanske,  for the very helpful conversations and detailed feedback. Thank you to John Voight for helpful suggestions on improving the proofs in the last section.


\begin{thebibliography}{9}

\bibitem{ants}
Angelica Babei. Computing normalizers of tiled orders in $M_n(k)$. In \textit{Proceedings of the Thirteenth Algorithmic Number
              Theory Symposium}, volume 2 of \textit{Open Book Series}, pages 55-68. Math. Sci. Publ. Berkeley, CA, 2019.

\bibitem{thesis} Angelica Babei. \textit{On the Arithmetic of Tiled Orders.} PhD thesis, Dartmouth College, Hanover, New Hampshire, USA, May 2019.

\bibitem{brown}
Kenneth S. Brown. \textit{Buildings}. Springer-Verlag, New York, 1989.

\bibitem{gentype}
Juliusz Brzezinski. A generalization of Eichler's trace formula. \textit{Collect. Math.} 48(1-2): 53--61, 1997.

\bibitem{deuring}
Max Deuring. Die Anzahl der Typen von Maximalordnungen einer defniten Quaternionenalgebra mit
primer Grundzahl. \textit{Jber. Deutsch. Math. Verein.} 54: 24--41, 1950.

\bibitem{zurzahlen}
Martin Eichler. Zur Zahlentheorie der Quaternionen-Algebren. \textit{J. Reine Angew. Math.} 195: 127--151, 1955.

\bibitem{tiledisom}
Hisaaki Fujita and Hiroshi Yoshimura. A criterion for isomorphic tiled orders over a local Dedekind
domain. \textit{Tsukuba Journal of Mathematics.} 16(1): 107--111, 1992.


\bibitem{lang}
Serge Lang. \textit{Algebraic number theory.} Volume 110 of \textit{Graduate texts in mathematics}, Springer-Verlag, New York, 1994.


\bibitem{local}
Benjamin Linowitz and Thomas R. Shemanske. Local selectivity of orders in central simple algebras.
\textit{Int. J. Number Theory}, 13(4): 853--884, 2017.


\bibitem{neukirch}
J\:{u}rgen Neukirch. \textit{Algebraic Number Theory.} Volume 322 of \textit{Grundlehren der mathematischen Wissenschaften.} Springer-Verlag Berlin Heidelberg, 1999.

\bibitem{ternare}
Meinhard Peters. Tern\:{a}re und quatern\:{a}re quadratische Formen und Quaternionenalgebren. \textit{Acta Arith.}
15: 329--365, 1968/1969.

\bibitem{pizer2}
Arnold Pizer. On the arithmetic of quaternion algebras. \textit{Acta Arith.} 31(1): 61--89, 1976.

\bibitem{theta}
Arnold Pizer. The representability of modular forms by theta series.  \textit{J. Math. Soc. Japan} 28(4): 689--698, 1976.

\bibitem{type}
Arnold Pizer. Type numbers of Eichler orders. \textit{J. Reine Angew. Math.} 264: 76--102, 1973.


\bibitem{graduated}
Wilhelm Plesken. \textit{Group-rings of finite-groups over p-adic integers.} Issue 1026 of  \textit{Lecture Notes in Mathematics.} Springer-Verlag Berlin Heidelberg, 1983.

\bibitem{reiner}
Irving Reiner. \textit{Maximal orders},  Issue 28 of \textit{London Mathematical Society monographs.} Clarendon Press,  Oxford, New York, 2003.

\bibitem{ronan}
Mark Ronan. \textit{Lectures on buildings.} University of Chicago Press, Chicago, IL, 2009. 

\bibitem{ktheory}
Jonathan Rosenberg. \textit{Algebraic K-theory and its applications}, volume 147 of \textit{Graduate Texts in Mathematics}. Springer-Verlag, New York, 1994.

\bibitem{normalizers}
Thomas R. Shemanske. \textit{Normalizers of graduated orders.} Preprint, 2016.


\bibitem{poly}
Thomas R. Shemanske. Split orders and convex polytopes in buildings. \textit{Journal of Number Theory.} 130(1):101--115, 2010.


\bibitem{lmfdb}
The LMFDB Collaboration. The L-functions and modular forms database, 2013.
\text{http://www.lmfdb.org}

\bibitem{sagemath}
The Sage Developers. SageMath, the Sage Mathematics Software System (Version 8.2), 2018.
\text{http://www.sagemath.org.}


\bibitem{shim}
 Goro Shimura.  Integer-Valued Quadratic Forms and Quadratic Diophantine Equations. \textit{Documenta Math.} 11: 333-367, 2006.

\bibitem{unram}
 Jude Socrates and David Whitehouse. Unramified Hilbert modular forms, with examples relating to elliptic curves. \textit{Pacific Journal of Mathematics} 219(2): 333--364, 2005.

\bibitem{arithmetique}
Marie-France Vign\'{e}ras. \textit{Arithm\'{e}tique des alg\'{e}ebres de quaternions}. Volume 800 of \textit{Lecture Notes in Mathematics.} Springer, Berlin, 1980.

\bibitem{quatbook}
John Voight. \textit{Quaternion algebras,.} v.0.9.14, 2018. \text{http://quatalg.org.}

\bibitem{zas}
Hans Zassenhaus. \textit{Graduated orders.} Unpublished notes, 1975.

\end{thebibliography}
\end{document}